%% file: tameness-general-groups_Fuhrmann_Kwietniak.tex
\documentclass[10pt, a4paper,reqno]{amsart}

\usepackage[utf8]{inputenc}
\usepackage[T1]{fontenc}
\usepackage{paralist,charter}
\usepackage{amssymb}
\usepackage{amsmath}
\usepackage{amsthm}
\usepackage[ngerman,english]{babel}
\usepackage[noadjust]{cite}
\usepackage[hang]{caption}
\usepackage{fancyhdr}
\usepackage{tikz}
\usepackage{mathrsfs}
\usepackage{enumitem}

\numberwithin{equation}{section}
\numberwithin{figure}{section}
\DeclareMathOperator{\Stab}{Stab}

\DeclareMathOperator{\gr}{gr}
\DeclareMathOperator{\idn}{id}
\newtheoremstyle{mythm}{3pt}{3pt}{\itshape}{0pt}{\bfseries}{.}{0.5eM}{}
\theoremstyle{mythm}

\newtheorem{definition}{Definition}[section]
\newtheorem{thm}[definition]{Theorem}

\newtheorem{lem}[definition]{Lemma}
\newtheorem{cor}[definition]{Corollary}
\newtheorem{prop}[definition]{Proposition}

 \newtheorem{claim}[definition]{Claim}

\newtheoremstyle{myrem}{3pt}{3pt}{\normalfont}{0pt}{\bfseries}{.}{0.5em}{}
\theoremstyle{myrem}

\newtheorem{defn}[definition]{Definition}
\newtheorem{rem}[definition]{Remark}

\newcounter{paranum}[section]

\newcommand{\mc}{\mathcal}
\newcommand{\w}{\omega}

\renewcommand{\:}{\colon}

\newcommand{\G}{\T}

\newcommand{\Seq}{S_{\text{eq}}}

\newcommand{\Tsim}{\T/\!\!\!\sim}

\newcommand{\SO}{\operatorname{SO(3)}}
\renewcommand{\S}{\mathbb S}

\input bibstandard

\title{On tameness of almost automorphic dynamical systems\\ for general groups} 
\author{Gabriel Fuhrmann$^1$} \address{$^1$Department of Mathematics, Imperial College London, 180 Queen’s Gate, London
  SW7 2AZ, UK. email: gabriel.fuhrmann@imperial.ac.uk} \author{Dominik Kwietniak$^2$}\address{$^2$Faculty of Mathematics and Computer Science, 
  Jagiellonian University in Krakow, ul. {\L}ojasiewicza 6, 30-348 Krak\'ow, Poland. email: dominik.kwietniak@uj.edu.pl}

\begin{document}
\begin{abstract}
 Let $(X,G)$ be a minimal equicontinuous dynamical system, where $X$ is a compact metric space
 and $G$ some topological group acting on $X$.
 Under very mild assumptions, we show that the class of regular almost automorphic extensions of $(X,G)$ contains examples of 
 tame but non-null systems as well as non-tame ones.
 To do that, we first study the representation of almost automorphic systems by means of semicocycles for general groups. 
 Based on this representation, we obtain examples of the above kind
 in well-studied families of group actions.
 These include Toeplitz flows over $G$-odometers where $G$ is countable and residually finite
 as well as symbolic extensions of irrational rotations.
\end{abstract}

\maketitle

The probabilistic concept of independence is at the heart of several fundamental notions in 
ergodic theory like ergodicity, mixing, or positive entropy.
Carried over to topological dynamics, independence gains a more combinatorial flavour and provides a basis for
the local analysis of topological entropy (initiated by Blanchard in \cite{Blanchard1993}) and related mixing properties,
see \cite{HLY2012}.  
For a comprehensive account of 
the combinatorial perspective on independence in topological dynamics with emphasis on entropy, 
see e.g., \cite{KerrLi2007, KerrLiBook}.
We study the absence of independence due to \emph{tameness}.

To gain some intuition, let us briefly discuss tameness for a binary subshift $(X,\sigma)$. 
% By a \emph{subshift} we mean $\Z$-action induced by the left-shift operator $\sigma$ on a closed $\sigma$-invariant 
% subset $X\ssq \{0,1\}^\Z$.
In this case, a set $J\ssq \Z$ is an \emph{independence set} for $X$ if for each $z\in \{0,1\}^J$ 
there is $x\in X$ with $x_j=z_j$ for every $j\in J$.
The study of independence sets is of fundamental importance as the existence or absence of large
independence sets implies
strong dynamical consequences \cite{KerrLi2007}. 
For example, the subshift $(X,\sigma)$ has positive topological entropy (as introduced
by Adler, Konheim, and McAndrew \cite{AdlerKonheimMcAndrew1965})
if and only if $X$ has an independence set of positive asymptotic density.
At the opposite end, $X$ has zero topological sequence entropy (as introduced by Goodman \cite{Goodman1974})
if and only if $X$ is a \emph{null system}, that is, there is a finite upper bound on the size of independence sets. % (this is also referred to as \emph{nullness}).
The lack of infinite independence sets is equivalent
to \emph{tameness}, a notion introduced to topological dynamics by K\"ohler \cite{Kohler1995}.

%\marginpar{\color{blue} Bourgain-Frmlin-Talagrand dichotomy?}

The last decade saw an increased interest in tame systems
(see e.g. \cite{Glasner2006,Glasner2007,GlasnerMegrelishvili2006,Huang2006,KerrLi2007,Romanov2016}; see
also \cite{GlasnerMegreshvili2018} for an up to date account)
revealing their connections to other areas of mathematics like Banach spaces \cite{GlasnerMegrelishvili2012},
circularly ordered systems \cite{GlasnerMegrelishvili2018Monatshefte}, substitutions and tilings, quasicrystals,
cut and project schemes and even model theory and logic
\cite{Aujogue2015,ChernikovSimon2018,GlasnerMegreshvili2013,Ibarlucia2016}.
A major breakthrough in the general understanding of tameness was achieved by Glasner's recent structural
result for tame minimal systems \cite{Glasner2018}.
One of its consequences is that a tame minimal dynamical system
which has an invariant measure is almost automorphic, uniquely ergodic and measure-theoretically isomorphic
to its maximal equicontinuous factor \cite[Corollary~5.4]{Glasner2018} (see also
\cite{Glasner2006,Huang2006,Glasner2007,KerrLi2007} for previous results in this direction).
In fact, a recent result shows that such systems are actually regularly almost automorphic, see
\cite[Theorem~1.2]{FuhrmannGlasnerJagerOertel2018}.

In view of these results, it is natural to ask whether there are non-tame regular almost automorphic extensions
of equicontinuous systems.
Further, as asked in \cite{Huang2006}: if a regular extension is tame, can it be non-null?
So far, 
few regular non-tame extensions are known (see \cite[Corollary~3.7]{FuhrmannGlasnerJagerOertel2018}) and
the only positive answer to the second question is provided by specific Toeplitz shifts
constructed in \cite[Chapter~11]{KerrLi2007}.
We show that the answer to both questions is emphatically \emph{yes}. 
In fact, under very mild assumptions any metric equicontinuous dynamical system $(\T,G)$ 
has almost one-to-one extensions which 
are non-tame as well as extensions which are tame but not null, 
see Theorem~\ref{thm:tame_non-null_example} and Theorem~\ref{thm: non-tame and non-null examples}.
The basis for our construction are so-called semicocycle extensions
which provide straightforward and flexible tools to
obtain a variety of examples of almost automorphic systems.

For $\Z$-actions, it is known that a dynamical system is a semicocycle extension of a group rotation
if and only if it
is an almost automorphic extension of the same rotation \cite{DownarowiczDurand2002} (provided the semicocycle is
\emph{invariant under no rotation}, see Section~\ref{sec: semi-cocycle extensions} for further details).
As a matter of fact, this observation and its proof immediately carry over to actions of abelian groups.
Our first goal is to extend this characterisation to general, non-abelian groups in Section~\ref{sec: semi-cocycle extensions}.
With this generalised notion of semicocycle extensions, we can rather directly construct a plethora of examples
of almost automorphic systems.
As an application, we obtain non-tame as well as tame but non-null symbolic systems such as Toeplitz flows over $G$-odometers 
with countable residually finite $G$ as well as
symbolic extensions of irrational rotations.
En passant, we obtain a generalisation (by completely different means) of the well-known fact that every 
minimal $\Z$-rotation on a compact metrizable monothetic group
allows for an almost automorphic symbolic
extension \cite[Theorem~3.1]{Paul1976}, see
Corollary~\ref{cor: symbolic non-tame extensions}.

\bigskip

%\noindent{\bf Acknowledgements.}
\subsection*{Acknowledgements}
This project has received funding from the European Union's Horizon 2020
research and innovation programme under the Marie Sk\l{}odowska-Curie grant
agreement No 750865.
The research of DK was supported by the National Science
Centre, Poland, grant no. 2018/29/B/ST1/01340.
The first ideas of this work were developed during a Research in Pairs stay (R1721) of the two authors
together with Maik Gr\"oger at the 
Mathematisches Forschungsinstitut Oberwolfach in October 2017.
The authors would like to thank the MFO for its hospitality and Maik Gr\"oger for the related discussions and 
by far not only mathematically entertaining time during those two weeks.

%%%%%%%%%%%%%%%%%%%%%%%%%%%%%%%%%%%%%%%%%%%%%%%%%%%%%%%%%%%%%%%%%%%%%%%%%%%%%%%%%%%%%%%%%%%%%%%%%%%%%%%%%%%%%%%%%%%%%%%%%%%
\section{Background in topological dynamics}\label{sec: background in topological dynamics}
The statements of this section as well as their proofs can be found in standard references on topological
dynamics and ergodic theory such as \cite{Auslander1988,Glasner2003}.
We say that a triple $(X,G,\Phi)$ is a \emph{topological dynamical system} if $G$ is a  topological group, 
$X$ is a compact Hausdorff topological space and $\Phi\colon G\times X\to X$ is a jointly continuous left action of $G$ on $X$.
Most of the time, we keep the action $\Phi$ implicit. 
That is, we simply refer to $(X,G)$ as a \emph{topological dynamical system} and write $gx$ for the image
$\Phi(g,x)$. 
Given $g\in G$, we refer to the homeomorphism $x\mapsto gx$ also by \emph{$g$-translation} and 
may identify an element $g\in G$ with that homeomorphism. 
We call the set $Gx=\{gx:g\in G\}$ the \emph{$G$-orbit} of $x$ or simply \emph{orbit} of $x\in X$ (under the action of $G$). 
The system $(X,G)$ is said to be \emph{minimal} if for each $x\in X$ the orbit of $x$ is dense in $X$,
that is, we have $\overline {Gx}=X$.

A topological dynamical system $(X,G)$ is {\em effective} if distinct elements $g$ and $g'$ of $G$ define different homeomorphisms,
that is, if for every $g,g'\in G$ with $g\neq g'$ there is $x\in X$ satisfying $gx\neq g'x$. 
We may always assume a system to be effective by identifying each element $g \in G$ with its $g$-translation as mentioned above.
We say that $G$ acts \emph{freely} on $x\in X$ if $gx\neq x$ for all $g\in G$ different from the neutral element
$e_G\in G$.
The dynamical system $(X,G)$ is \emph{free} if $G$ acts freely on every $x\in X$.
It is well known and straightforward to see that if $G$ is abelian and acts minimally on $X$, then $(X,G)$ is
free if and only if $(X,G)$ is effective.

 A topological dynamical system $(X,G)$ is \emph{equicontinuous} if the collection of $g$-translations
 $\{x\mapsto gx : g\in G\}\ssq X^X$
 is a family of maps from $X$ to $X$ which is equicontinuous (with respect to the unique uniformity $\mathscr{U}_X$ that generates the topology on $X$). 
 In this case, we have that for every $\alpha\in\mathscr{U}_X$ there is $\beta\in\mathscr{U}_X$ such that whenever 
 $(x,x')\in\beta$ and $g\in G$ we have $(gx,gx')\in\alpha$. 
 If $(X,G)$ is an equicontinuous topological dynamical system and if $X$ is metrizable, we can choose a 
 compatible metric on $X$ such that each $g$-translation $x\mapsto gx$ is an isometry with
 respect to this metric. 
 For that reason, whenever $X$ is metrizable, we will use the terms equicontinuous and isometric synonymously.
 
 Recall that an \emph{invariant measure} of a topological dynamical system $(X,G)$
 (or: a $G$-\emph{invariant measure} on $X$)
 is a Radon probability measure $\mu$ on $X$ such that
 $\mu(gA)=\mu(A)$ for all $g\in G$ and all $A\in \mathscr B(X)$, where $\mathscr B(X)$ denotes the collection of all
 Borel sets of $X$.
%  We denote the collection of all measurable sets by $\mathscr B(X)$ and throughout
%  assume the measure space $(X,\mathscr B(X),\mu)$ to be complete.
 Given an invariant measure $\mu$, a set $A\in \mathscr B(X)$ is called 
 \emph{invariant (with respect to $\mu$)} if for all $g\in G$ we have
 $\mu(A\,\triangle\, g A)=0$.
 If $\mu$ is an invariant measure such that 
 for every invariant set $A\in \mathscr B(X)$ we have that
 either $\mu(A)=0$ or $\mu(A)=1$, then $\mu$ is referred to as being \emph{ergodic}.
 It is well-known that every minimal equicontinuous system 
 is \emph{uniquely ergodic}, that is, it allows for a unique invariant measure
 which is necessarily ergodic (see also Theorem~\ref{thm: representation of equicont sys}).
 
Let $(X,G)$ and $(Y,G)$ be two topological dynamical systems (with the same acting group $G$). 
A \emph{homomorphism} from $(X,G)$ to $(Y,G)$ is a continuous map $\pi\colon X\to Y$ such that
for every $x\in X$ and $g\in G$ we have $g\pi(x)=\pi(gx)$.
If there is a homomorphism $\pi\colon X\to Y$ which is an \emph{onto} map, then we say that $(Y,G)$ is a
\emph{factor} of $(X,G)$, $(X,G)$ is an \emph{extension} of $(Y,G)$, and that $\pi$ is an
\emph{epimorphism} or a \emph{factor map}. 
In the above situation, the terms \emph{isomorphism}, 
\emph{automorphism}, and \emph{endomorphism} and accordingly, the notion of two systems being \emph{isomorphic} 
have their standard meaning. 
A minimal topological dynamical system $(X,G)$ is \emph{coalescent} if all its endomorphisms are
automorphisms. 
Minimal equicontinuous systems are always coalescent (see, \cite[page 81]{Auslander1988}).
Further, 
factors of minimal systems are minimal and
factors of equicontinuous systems are equicontinuous (\cite[Corollary 2.6]{Auslander1988}). 

Note that if $\pi\colon X\to Y$ is a factor map, then
$R(\pi) = R = \{(x,x')\in X\times X: \pi(x)=\pi(x')\}$ is an \emph{invariant}, \emph{closed equivalence relation} 
(\emph{icer}) on $X$. 
That is to say, the equivalence relation $R$ is a closed subset of $X\times X$ and if
$(x,x') \in R$ and $g\in G$, then $(gx,gx') \in R$.  
Conversely, if $(X,G)$ is a topological dynamical system and $R$ is an icer
on $X$, then the quotient space $X/R$ is a compact Hausdorff
space. 
Furthermore, if $\pi\: X\to X/R$ is the corresponding quotient map,
then $\pi$ is an epimorphism from $(X,G)$ to $(X/R,G)$, where for all $g\in G$ the $g$-translation 
on $X/R$ is given by $\pi(x)\mapsto \pi(gx)$.
Hence, factor maps and icers are just two ways of talking about 
the same thing and we will use them interchangeably.

If $(X,G)$ is a topological dynamical system, there is a smallest icer $\Seq$ 
(known as the {\em equicontinuous structure relation}) such that the factor system
$(X/\Seq, G)$ is equicontinuous (see, for example, \cite[Theorem 9.1]{Auslander1988}).
We refer to $(X/\Seq, G)$ as well as to every system isomorphic to $(X/\Seq, G)$
as the {\em maximal equicontinuous factor} of $(X,G)$. 

We say that a factor map $\pi\colon X\to Y$ is \emph{almost one-to-one} if the set
\begin{align}\label{eq: defn almost automorphic points}
X_0= \{ x\in X\: \pi^{-1}(\{\pi(x)\}) = \{x\} \}
 \end{align}
is dense in $X$. 
In this case, we call the system $(X,G)$ an \emph{almost one-to-one extension} of $(Y,G)$.
A topological dynamical system is called \emph{almost automorphic} if its maximal equicontinuous factor is minimal
and the corresponding factor map $\pi$ is almost one-to-one.
In this case, we call points $x\in X$ with $\pi^{-1}(\{\pi(x)\}) = \{x\}$ \emph{almost automorphic}.
Observe that almost automorphic systems are necessarily minimal.
If the projection $\pi(X_0)$ of the almost automorphic points to the maximal equicontinuous factor
is measurable and of full measure (with
respect to the unique invariant measure on $\pi(X)$), we say
$(X,G)$ is \emph{regular}.
Clearly, every system isomorphic to an almost automorphic system is almost automorphic itself.

\begin{prop}[{cf. \cite[Proposition~1.1]{Paul1976}},
{\cite[V(6.1)5, page 480]{deVries1993}}] \label{prop:almost_automorphic}% factor is MEF}
 If $(X,G)$ is a minimal topological dynamical system, then the following statements are equivalent.
  \begin{enumerate}[label=(\emph{\alph*})]
 \item $(X,G)$ is almost automorphic.
 \item $(X,G)$ is an almost one-to-one extension of an equicontinuous system.
 \end{enumerate}
Furthermore, if $(X,G)$ is an almost one-to-one extension of a minimal equicontinuous system $(\T,G)$, 
then $(\T,G)$ is the maximal equicontinuous factor of $(X,G)$.
\end{prop}

\subsection{The Ellis semigroup and equicontinuous systems}
By $X^X$ we denote the collection
of all (not necessarily continuous) maps from $X$ to itself. We endow $X^X$ with the product topology which coincides with the topology of
pointwise convergence (a net $(\xi_n)$ in $X^X$ converges to $\xi\in X^X$ if and only if $\xi_n(x)\to\xi(x)$ for every $x\in X$). By Tychonoff’s theorem, $X^X$ is a compact
Hausdorff space. Furthermore, $X^X$ has a semigroup structure defined by composition of maps. 

Given a topological dynamical system $(X,G)$, the \emph{Ellis semigroup} $E(X)$ associated to $(X,G)$ is defined as the closure of the set of $g$-translations
$\{x\mapsto gx :g\in G\}$ in the space $X^X$. 
We may take the liberty to consider elements of $G$ as elements in $E(\T)$.
Note that, in general, there may be elements in $E(X)$ which are neither bijective nor continuous.

 \begin{thm}[{cf. \cite[Theorem~7, p.~54]{Auslander1988}}]\label{thm: ellis groups are homomorphic}
 Let $\pi\colon X\to Y$ be a factor map between two topological dynamical systems $(X,G)$ and $(Y,G)$.
 Then there exists a unique continuous semigroup epimorphism $\Phi\colon E(X)\to E(Y)$
 such that $\pi(\xi x)=\Phi(\xi)\pi(x)$ holds for every $x\in X$ and $\xi\in E(X)$.
\end{thm}

Note that there is a natural left action of $G$ on $E(X)$ given by  $E(X)\ni\xi\mapsto g\xi\in E(X)$ for each  $g\in G$. 
Clearly, $(E(X),G)$ is a topological dynamical system.

 \begin{thm}[{cf. \cite[pp.~52--53]{Auslander1988}}]\label{thm: representation of equicont sys}
 Suppose $(\G,G)$ is a minimal equicontinuous dynamical system.
 Then $E(\G)$ is a compact Hausdorff topological group, each $\xi\in E(\T)$ is a homeomorphism on $\T$,
 and $(E(\G),G)$ is a minimal 
 equicontinuous dynamical system, too. 
 There is also a jointly continuous  action of $E(\G)$ on $\G$ extending the action of $G$ on 
 $\G$ so that $(\G,E(\G))$ is a minimal equicontinuous dynamical system.
 If $\T$ is metrizable, then so is $E(\T)$.
 We further have:
 \begin{enumerate}[label=(\emph{\alph*})]
 \item The system $(\G,G)$ is a factor of $(E(\G),G)$ and for every $\theta\in \G$ the map $p_{\theta}\colon E(\G)\to \G$ 
 given by
 \[
  p_{\theta}(\xi)= \xi \theta
 \]
 is a factor map. 
 Furthermore, let $\Stab(\theta)=\{\xi\in E(\T)\colon\xi\theta=\theta\}$ be the stabiliser of
 $\theta \in \T$ with respect to the action of  $E(\T)$ on $\G$.  
 Then $\Stab(\theta)$ is a closed subgroup of $E(\T)$ and $(E(\T)/\Stab(\theta),G)$
 is isomorphic to $(\T,G)$. 
 In particular, $p_\theta$ is an open map and the push-forward of the Haar measure 
 on $E(\T)$ through the projection onto $E(\T)/\Stab(\theta)$ gives the unique invariant measure $m_\T$ of $(\T,G)$.
 \item If $G$ is abelian, then $E(\G)$ is abelian as well and $(\G,G)$ is isomorphic to $(E(\G),G)$.
 \end{enumerate}
\end{thm}

\begin{rem}\label{rem: action of ET on T}
 For later reference, let us briefly collect some properties of the action of $E(\T)$ on $\T$ provided by the above statement.
 First, note that an immediate consequence of the minimality of $(\T,G)$ is that $E(\T)$ acts \emph{transitively}
 on $\T$, that is, for each pair $\theta_1,\theta_2\in\T$ there is $\xi\in E(\T)$ with 
 $\xi\theta_1=\theta_2$.\footnote{The notion of transitivity should not be confused with the notions of 
 \emph{topological} or \emph{point transitivity} (see \cite[p.31]{Auslander1988}) which are commonly referred to by the abbreviated term \emph{transitive}, too.
 We would like to stress that throughout this work, we refer by transitive solely to the above concept.}
 
 Secondly, notice that the unique $G$-invariant measure $m_\T$ on $\T$ 
 necessarily coincides with the unique invariant measure of the minimal and equicontinuous system
 $(\T,E(\T))$.
%  This follows immediately from the obvious fact that every $E(\T)$-invariant
%  measure on $\T$ is necessarily $G$-invariant.

 Finally, observe that if $\T$ is metrizable and $\rho$ is a metric on $\T$ with respect to which $G$ acts isometrically
 on $\T$, then the action of $E(\T)$ on $\T$ is also isometric with respect to $\rho$.
\end{rem}
\begin{rem}
 If we have that $\Stab(\theta)$ is trivial, then Theorem~\ref{thm: representation of equicont sys} (a) yields that
 $(\T,G)$ is actually isomorphic to $(E(\T),G)$ and hence free, provided $G$ acts effectively on $\T$.
 The assumption of an action acting freely (in a weak sense) 
 will enter our constructions of tame non-null systems in Section~\ref{seq: examples}.
\end{rem}

\begin{cor}\label{cor: thetan->theta and xin->e}
 Suppose $(\G,G)$ is a minimal equicontinuous dynamical system and $(\theta_n)$ is a net in $\T$ with
 $\theta_n\to \theta$ for some $\theta \in \T$.
 Then there is a subnet $(\theta_m)$ of $(\theta_n)$ and a net $(\xi_m)$ in $E(\T)$
 with $\xi_m \theta_m=\theta$ and $\xi_m\to e$, where $e$ denotes the neutral element in the group $E(\T)$.
\end{cor}
\begin{proof}
Let $p_\theta\colon E(\G)\to\G$ be the factor map determined by $\theta$ (see Theorem~\ref{thm: representation of equicont sys}). 
For each $n$, choose $\hat \theta_n\in p_\theta^{-1}(\theta_n)$.
By compactness of $E(\T)$, there is a subnet $(\hat \theta_m)$ of $(\hat \theta_n)$ with
 $\hat \theta_m\to \hat \theta$ for some $\hat \theta$ which lies in $p_\theta^{-1}(\theta)\ssq E(\G)$, 
 due to the continuity of $p_\theta$.
Let $\xi_m=\hat \theta\hat\theta_m^{-1}$.
 Since $E(\T)$ is a topological group, we have $\xi_m\to e$.
 At the same time, it holds that $\xi_m\theta_m=\xi_m p_\theta(\hat \theta_m)=\xi_m (\hat \theta_m\theta)=
 \xi_m \hat \theta_m\theta=\hat \theta \theta=p_\theta(\hat \theta)=\theta$.
\end{proof}

\begin{cor}\label{cor: stabilizers have empty interior}
 Suppose $(\G,G)$ is a minimal equicontinuous dynamical system
 and $\T$ is infinite.
 Let $\theta,\theta'\in \T$.
 Then in each neighbourhood of the neutral element $e \in E(\T)$ there is $g\in G$ (considered
 as an element of $E(\T)$) with $g\notin \Stab(\theta)\cup \Stab(\theta')$, that is,
 $g\theta\neq \theta$ and $g\theta'\neq\theta'$.
\end{cor}
\begin{proof}
 Clearly, every subgroup of $E(\T)$ which has non-empty interior (in $E(\T)$) is, in fact, open (in $E(\T)$).
 Further, it is straightforward to see and well known that an open subgroup of a compact group
 is necessarily of finite index.
 As $\T$ is infinite and
 homeomorphic to $E(\T)/\Stab(\theta)$ and $E(\T)/\Stab(\theta')$ (due to Theorem~\ref{thm: representation of equicont sys}),
 we clearly have that neither $\Stab(\theta)$ nor $\Stab(\theta')$ has finite index.
 Hence, $\Stab(\theta)$ and $\Stab(\theta')$ have empty interior. 
 
  As $\Stab(\theta)$ and $\Stab(\theta')$ are closed, we thus have that the complement of
  $\Stab(\theta)\cup\Stab(\theta')$ is open and dense.
 As $G$ embeds densely in $E(\T)$ (by definition), the statement follows. 
%  Therefore, given an arbitrary neighbourhood $U$ of $e$, there is $\xi\in U\setminus \Stab(\theta)$.
%  As $\Stab(\theta)$ is closed, there actually is an open neighbourhood $V\ssq U$ of $\xi$ with
%  $\Stab(\theta)\cap V=\emptyset$.
%  By a similar reasoning, there is an open set $V'\ssq V$ with $\Stab(\theta')\cap V'=\emptyset$ (and clearly
%  $\Stab(\theta)\cap V'=\emptyset$).
%  As $G$ embeds densely in $E(\T)$ (by definition), there is $g\in G\cap V$.
%  This finishes the proof.
\end{proof}

\section{Semicocycle extensions}\label{sec: semi-cocycle extensions}
In this section, we introduce a representation of almost automorphic systems 
by means of Sturmian-like subshifts with a compact alphabet.
This construction builds upon so-called \emph{semicocycles}.
For $G=\Z$, these have already proved useful in the study of factors of Toeplitz shifts
(see \cite{DownarowiczDurand2002,Downarowicz2005}) but also of almost one-to-one extensions of  
non-equicontinuous systems (see \cite{DownarowiczSerafin2005} for a nice exposition).
For symbolic $\Z$-shifts, similar techniques can be found in \cite{Paul1976,MarkleyPaul1976}.
Here, we extend the constructive idea of semicocycles to general groups, in particular, to non-abelian ones.
We would like to mention that although the underlying ideas are close to those for semicocycle extensions
of $\Z$-actions some care has to be taken in the course of this generalisation.

We will frequently deal with maps $f\: \T_0\subseteq \T\to K$, where $\T$ and $K$ are Hausdorff topological spaces
and $\T_0$ is some subset of $\T$.
Despite the fact that such $f$ may not be defined everywhere in $\T$, we will always 
consider the graph of $f$ as a subset of
$\T\times K$.
That is, we will refer by $\gr f$ to the set 
$\{(\theta,k)\in \T\times K\:\theta\in \T_0,k=f(\theta)\}\ssq \T\times K$.
Moreover, we denote by $F\ssq \G\times K$ the closure of the graph of $f$ as a 
subset of $\T\times K$, that is, $F=\overline{\gr f}$.
Finally, for every $\theta\in\G$, we define the \emph{$\theta$-section of $F$} as the
set $F(\theta)=\{k\in K\: (\theta,k)\in F\}$.
The proof of the following statement is straightforward and left to the reader.
\begin{prop}\label{prop:graphs}
Assume that $\T$ and $K$ are Hausdorff topological spaces and $K$ is compact.
Let $\T_0\ssq\T$ be dense, $f\colon \T_0\ssq \T\to K$ be a mapping, and
$F=\overline{\gr f}\ssq\T\times K$.
\begin{enumerate}
\item Let $\theta\in\T_0$. 
     The function $f$ is continuous at $\theta$ (with respect to the subspace topology on $\T_0$ inherited from $\T$)
     if and only if  $F(\theta)=\{f(\theta)\}$.
\item If $f$ is continuous at every $\theta\in\T_0$, %If $F(\theta)=\{f(\theta)\}$ for every $\theta\in\T_0$,
	then for every dense set $\T_1\ssq \T$ and every function $g\colon \T_1\to K$ with
	$g(\theta')\in F(\theta')$ for all $\theta'\in \T_1$,
	the set $\gr g$ is dense in $F$.
\end{enumerate}
\end{prop}

For the rest of this work, $K$ is always assumed to be a compact Hausdorff space.
We say that $\theta\in\G$ is a \emph{discontinuity point} of $f$ if $F(\theta)$ has more than one element.
We write  
\[
    D_f=\{\theta\in \T : \# F(\theta)>1\}\ssq \T
\]
for the set of all discontinuity points of $f$.
If $f\:\T_0\to \T$ is continuous, we clearly have $\T_0\cap D_f=\emptyset$.
In this case, if $\T_0$ is further dense in $\T$, the mapping
\begin{align}\label{eq: semicocycle extended to complement of Df}
\T\setminus D_f  \to K,  \qquad \theta\mapsto k_\theta,
\end{align}
where $k_\theta$ is such that $F(\theta)=\{k_\theta\}$,
is a well-defined continuous extension of $f$.
For simplicity, we may refer to this mapping by $f$ as well.

Recall that a triple $(\T,G,\theta_0)$ is a \emph{pointed dynamical system} if
$(\T,G)$ is a topological dynamical system and $\theta_0$ is an element in $\T$ with $\overline{G\theta_0}=\T$.
A \emph{($K$-valued) semicocycle} over a pointed dynamical system $(\T,G,\theta_0)$ is a map 
$f\colon G\theta_0\to K$ which is continuous with respect to the subspace
topology on $G\theta_0\ssq \T$.
If $f$ is a semicocycle, then for every $\theta\in\G$ the section $F(\theta)$ is nonempty and
$F(g\theta_0)=\{f(g\theta_0)\}$ for every $g\in G$ (see Proposition~\ref{prop:graphs}).

Given a $K$-valued semicocycle $f$ over a pointed system $(\T,G,\theta_0)$, we now define
a topological dynamical system $(X_f,G)$ associated to $f$.
To that end, observe that $G$ acts from the left
on the product topological space $K^G$ by means of the \emph{shift action}
\[G\times K^G\ni(h,(x_g)_{g\in G})\mapsto \sigma^h((x_g)_{g\in G})=(x_{gh})_{g\in G}\in K^G.\]
By slightly abusing notation, we may identify $f$ with the mapping $g\mapsto f(g\theta_0)$ on $G$ and hence
consider $f$ an element of $K^G$.
We define $(X_f,G)$ as the orbit closure of $f$ under the above shift action.
Throughout this work, given $x\in K^G$, we synonymously refer by $x_g$ and $x(g)$ to the image of $g$ under $x$.

It is natural to ask whether $(X_f,G)$ is an extension of $(\T,G)$.
Towards an answer to this question, we introduce an equivalence relation $\sim$ on $\T$ by putting
\begin{equation}\label{def:sim}
 \theta_1 \sim \theta_2 \text{ if and only if for every } \xi \in E(\T)\text{ we have } F(\xi \theta_1)=F(\xi \theta_2),
\end{equation}
where $E(\T)$ denotes the Ellis semigroup of $(\T,G)$.
We denote the equivalence class of $\theta\in \T$ by $[\theta]$ and say the semicocycle 
$f$ is \emph{invariant under no rotation} if
$\sim$ is the identity relation, that is, if $\#[\theta]=1$ for every $\theta\in\T$. 

\begin{lem}\label{lem:sim is an icer}
If $(\T,G)$ is a minimal equicontinuous system, then the relation $\sim$ is an icer. 
\end{lem}
\begin{proof} 
 Since $G$ embeds in $E(\T)$, we easily see that $\sim$ is $G$-invariant.
 It remains to show that $\sim$ is closed in $\T\times \T$.
 Let $(\theta_1,\theta_2)\in \T^2$ be given and suppose there is a net $(\theta_1^n,\theta_2^n)\to(\theta_1,\theta_2)$
 with $\theta_1^n\sim\theta_2^n$.
 Due to Corollary~\ref{cor: thetan->theta and xin->e}, we may assume without loss of generality that there is a 
 net $(\xi_n)$ in $E(\T)$ with $\xi_n\theta_1^n=\theta_1$ and $\xi_n\to e$.
 Then $F(\xi\theta_1)=F(\xi\xi_n\theta_1^n)=F(\xi \xi_n\theta_2^n)$ for every $\xi\in E(\T)$.
 Since $\xi\xi_n\theta_2^n\to\xi\theta_2$ %\footnote[2]{Recall that elements of $E(\T)$ are continuous on $\G$ by Theorem~\ref{thm: representation of equicont sys}.} 
 and $F$ is closed, this yields $F(\xi\theta_1)\ssq F(\xi\theta_2)$ for each $\xi\in E(\T)$.
 Interchanging the roles of $\theta_1$ and $\theta_2$, we obtain $F(\xi\theta_2)\ssq F(\xi\theta_1)$ for every $\xi\in E(\T)$ and hence $\theta_1\sim\theta_2$.
\end{proof}
If not stated otherwise, we throughout assume the system $(\T,G)$ to be minimal and equicontinuous.
In this case, given some $\theta_0\in \T$ and a semicocycle $f$ over $(\T,G,\theta_0)$
which is invariant under no rotation,
we refer to the system $(X_f,G)$ as a \emph{semicocycle extension of $(\T,G)$}.
As we will see in Theorem~\ref{thm: equivalence semicocylce extension and almost automorphic system},
a semicocycle extension of $(\T,G)$ is indeed an extension of $(\T,G)$.
The following statement provides a simple but useful sufficient criterion for 
invariance under no rotation.
\begin{prop}\label{prop: criterion imnvariance under no rotation}
 Let $f$ be a semicocycle over $(\T,G,\theta_0)$.
 If for each $\theta_1,\theta_2\in \T$ there is $\xi \in E(\T)$ such that
 $\xi\theta_1 \in D_{f}$ and $\xi\theta_2\notin D_f$, then $f$ is invariant under no rotation.
\end{prop}
\begin{proof}
 This immediately follows from the fact that, by definition of $D_f$, we have
 \[
  \# F(\xi\theta_1)> 1=\#F(\xi\theta_2). \qedhere
 \]
\end{proof}

The next statement suggests that we may realise invariance under no rotation by possibly changing the 
pointed dynamical system (see also \cite[Theorem~6.5]{Downarowicz2005} for a similar statement
for semicocycles over $\Z$-odometers).
\begin{lem}\label{lem: ensure invariance under no rotation}
 Suppose $(\T,G)$ is a minimal equicontinuous system, $\theta_0\in \T$
 and $f$ is a semicocycle over $(\T,G,\theta_0)$.
 Then $(X_f,G)$ is a semicocycle extension of $(\T/\!\!\sim,G)$.
\end{lem}
\begin{proof}
 Since $\sim$ is an icer on $\T$, we have
 that $(\T/\!\!\!\sim,G)$ is a (necessarily minimal and equicontinuous) factor of $(\T,G)$\footnote[3]{We would like to remark that
 if $\T$ is metrizable, Stone's Metrization Theorem yields that $\T/\!\!\!\sim$ is also metrizable \cite[Theorem~1]{Stone1956}.}
 so that $(\T/\!\!\!\sim,G,[\theta_0])$ is clearly a pointed dynamical system.
 Further, note that $\tilde f\: G[\theta_0]\ssq \T/\!\!\!\sim\, \to K$ given by $g[\theta_0]\mapsto f(g\theta_0)$
 is well-defined and that $\tilde F([\theta])=F(\theta)$ for all $\theta\in \T$.
 It follows that $\tilde f$ is, in fact, continuous,
 by Proposition~\ref{prop:graphs}~(1).
 With the epimorphism $\Phi$ from Theorem~\ref{thm: ellis groups are homomorphic}, we obtain
 $\tilde F(\Phi(\xi)[\theta])=\tilde F([\xi\theta])=F(\xi\theta)$ for all $\xi \in E(\T)$
 which implies that $\tilde f$ is invariant under no rotation.
\end{proof}

The next result is a generalisation of the respective statement for $\Z$-actions
(cf. \cite[Theorem~6.4]{Downarowicz2005} and \cite[Theorem~5.2]{DownarowiczDurand2002}).
The main idea of its proof is as in these references.
However, as the group $G$ is not necessarily abelian, the homogeneous space $\T$ may not possess a 
compatible group structure
(see Theorem~\ref{thm: representation of equicont sys}).
This fact requires an extra passage through the Ellis semigroup $E(\T)$ of $\T$
(see, in particular, the proof of equation \eqref{claim:invariance}).

\begin{thm}\label{thm: equivalence semicocylce extension and almost automorphic system}
Let $(X,G)$ and $(\T,G)$ be topological dynamical systems and assume $(\T,G)$ is minimal and equicontinuous.
The following statements are equivalent.
\begin{enumerate}
\item\label{cond:a} $(X,G)$ is an almost automorphic extension of $(\T,G)$.
\item\label{cond:b} $(X,G)$ is topologically isomorphic to a semicocycle extension $(X_f,G)$ of $(\T,G)$.
\end{enumerate}
\end{thm}
\begin{rem}
 Note that due to Proposition~\ref{prop:almost_automorphic}, the above statement yields that
 $(\T,G)$ is the maximal equicontinuous factor of any of its semicocycle extensions.
\end{rem}

\begin{proof}[{Proof of Theorem~\ref{thm: equivalence semicocylce extension and almost automorphic system}}]
We first show that $\eqref{cond:b}$ implies \eqref{cond:a}. 
Let $f$ be a semicocycle over the pointed minimal equicontinuous system $(\T,G,\theta_0)$ (where $\theta_0 \in \T$)
and assume $f$ is
invariant under no rotation.
Note that by Proposition~\ref{prop:almost_automorphic} it suffices to show that $(X_f,G)$ is
an almost one-to-one extension of $(\T,G)$.
% To that end, we define a map $\pi$ on a dense subset of $X_f$ by setting
% $\pi(\sigma^g(f(\cdot\, \theta_0)))=g\theta_0\in \T$ for all $g\in G$.

Take $x=(x(g))_{g\in G}\in X_f$. 
By definition, there is a net $(h_n)$ in $G$ with $\lim_n f(gh_n\theta_0)=x(g)$ for each $g\in G$.
It is natural to take some accumulation point $\theta_x\in\T$ of the net $(h_n\theta_0)$ 
and define
\begin{align}\label{eq: defn of factor pi}
 \pi\: X_f \to \T, \qquad x\mapsto \theta_x.
\end{align}
Observe that for every $g\in G$ we necessarily have
\begin{align}\label{eq: coordinate entries lie in sections}
    x(g)\in F(g\theta_x)=F(g\pi(x)),
\end{align}
where $F=\overline{\gr f}\subset \T\times K$ as above.
Our goal is to prove that $\pi$ is, in fact, an almost one-to-one factor map from $X_f$ to $\T$.

To that end, we first show that $\pi$ is uniquely defined.
Note that this will yield that $\pi$ is continuous. 
Take two accumulation points $\theta^1_x,\theta^2_x\in\T$ of the net $(h_n\theta_0)$. 
We will prove that
\begin{equation}
\label{claim:invariance}
F(\xi\theta^1_x)=F(\xi\theta^2_x) \quad\text{ for every }\xi\in E(\T).
\end{equation}
Since $f$ is invariant under no rotation, this yields that, in fact, $\theta^1_x=\theta^2_x$.

To show \eqref{claim:invariance}, we introduce some notation.
Let $p\colon E(\T)\to\T$ be a factor map as given by Theorem~\ref{thm: representation of equicont sys}.
Fix $\hat\theta^1_x,\hat\theta^2_x\in E(\T)$ with $p(\hat\theta^1_x)=\theta^1_x$ and 
$p(\hat\theta^2_x)=\theta^2_x$. 
Further, let 
\[\hat{F}=(p\times\idn_K)^{-1}(F)=\{(\xi,y)\in E(\T)\times K: (p(\xi),y)\in F\}.\]
Clearly, if $p(\xi)=p(\xi')$ for $\xi,\xi'\in E(\T)$, then $\hat F(\xi)=\hat F(\xi')$.

Since $p$ is an open map, the set $E_0=p^{-1}(G\theta_0)$ is dense in $E(\T)$.
Similarly, as $\hat f=f\circ p\colon E_0\to K$ verifies
$\gr\hat f=(p\times\idn_K)^{-1}(\gr f)$, we have that $\gr\hat f$ is dense in $\hat F$.
Observe that $\hat f$ is continuous.
Now, Proposition~\ref{prop:graphs}~(2) and equation \eqref{eq: coordinate entries lie in sections} yield that
$\Theta^1=\{(g\theta^1_x,x(g))\in\T\times K:g\in G\}$ is dense in $F$ and that
$\hat\Theta^1=\{(g\hat\theta^1_x,x(g))\in E(\T)\times K:g\in G\}$ is dense in $\hat{F}$.
We are ready to prove \eqref{claim:invariance}.

Fix $\xi_0\in E(\T)$ and take $y\in F(\xi_0\theta^1_x)$.
Clearly, $p(\xi_0\hat \theta^1_x)=\xi_0 p(\hat \theta^1_x)=\xi_0 \theta^1_x$ so that $(\xi_0\hat\theta^1_x,y)\in \hat F$. 
Due to the denseness of $\hat\Theta^1$ in $\hat F$, there is a net $(g_n)$ such that
$\lim_n (g_n\hat\theta^1_x,x(g_n))=(\xi_0\hat\theta^1_x,y)$, 
in particular, $\lim_n g_n\hat\theta^1_x=\xi_0\hat\theta^1_x$ so that $\lim_n g_n=\xi_0$ since $E(\T)$ is a group.
It follows that $\lim_n (g_n\hat\theta^2_x,x(g_n))=(\xi_0\hat\theta^2_x,y)$. 
Therefore, $y\in \hat F(\xi_0\hat\theta^2_x)=F(\xi_0p(\hat\theta^2_x))=F(\xi_0\theta^2_x)$. 
This proves that $F(\xi_0\theta^1_x)\ssq F(\xi_0\theta^2_x)$. 
Interchanging the roles of $\theta^1_x$ and  $\theta^2_x$, we obtain the reverse inclusion and hence
\eqref{claim:invariance}.

The map $\pi$ in \eqref{eq: defn of factor pi} is thus uniquely defined and continuous.
Further, we clearly have
 \[\pi(\sigma^s (f(gh\theta_0))_{g\in G})=\pi((f(gsh\theta_0))_{g\in G})=sh\theta_0=s \pi((f(gh\theta_0))_{g\in G})\]
for all $s,h\in G$.
Hence, by continuity of $\pi$ and denseness of $\{\sigma^hf(\cdot\,\theta_0)\: h\in G\}$ in $X_f$, we have $\pi(\sigma^sx)=s \pi(x)$
for all $x\in X_f$.
Thus, $(\T,G)$ is a factor of $(X_f,G)$.
Further, note that \eqref{eq: coordinate entries lie in sections} implies that $\pi$ is almost one-to-one which
hence proves (2).

Finally, we prove that \eqref{cond:a} implies \eqref{cond:b}. 
To this end, suppose we have an almost automorphic system $(K,G)$ with the maximal equicontinuous factor $(\T,G)$ and 
let $\pi\colon K\to\T$ be the associated factor map. 
Take an almost automorphic point $k\in K$ and define $f$ on $G\pi(k)\ssq \T$ by
$f(g\pi(k))= gk\in K$. 
Since $k$ is an almost automorphic point, it is not hard to see that $f$ is a continuous map.
Thus, $f$ is a $K$-valued semicocycle over $(\T,G,\pi(k))$. Let $X_f=\overline{\{f(gh\pi(k))_{g\in G}\:h\in G\}}\ssq K^G$. 
It remains to show that $(X_f,G)$ is isomorphic to $(X ,G)$ and that $f$ is invariant under no rotation, that is $\sim$ is a trivial equivalence relation on $\T$.
We begin by noting that the map $\psi\colon X_f\ni x\mapsto x(e)\in K$, where $e=e_{G}$ is the neutral element of $G$,  is onto and continuous.
To show that $\psi$ is injective, take $g,h\in G$ and observe that $f(gh\pi(k))=ghk=gf(eh\pi(k))$. 
This implies $x_{g}=gx_e$ for each $(x_g)_{g\in G}\in X_f$.
Thus, if $x_e=y_e$, then $(x_g)_{g\in G}=(y_g)_{g\in G}$, proving injectivity.
 Moreover, it is easy to see that $\psi(\sigma^g x)=g\psi(x)$ for every $g\in G$.
 Hence, $(X_f,G)$ is isomorphic to $(K,G)$ so that $(\T,G)$ is a maximal equicontinuous factor of $(X_f,G)$.
 At the same time, Lemma~\ref{lem: ensure invariance under no rotation} and
 the fact that \eqref{cond:b} implies \eqref{cond:a} yield that $(X_f,G)$ is an almost
 automorphic extension of $(\Tsim,G)$ which is --due to Proposition~\ref{prop:almost_automorphic}-- 
 a maximal equicontinuous factor of $(X_f,G)$, too.
 Therefore, we may consider the factor map $\T\ni \xi\to [\xi]\in\Tsim$ as an endomorphism of $(\T,G)$. 
 Due to the coalescence of equicontinuous systems,
 the factor map $\T\ni \xi\to [\xi]\in\Tsim$ must be a homeomorphism, so that $\sim$ is the identity relation.
 Thus, $(X_f,G)$ is a semicocycle extension of $(\T,G)$ isomorphic to $(K,G)$.
\end{proof}

\begin{rem}
 Observe that a semicocycle extension $(X_f,G)$ of $(\T,G)$ is regular
 if and only if $GD_f$ is measurable and $m_\T(GD_f)=0$.
\end{rem}

\subsection{A brief discussion of a weak form of freeness.}
In order to obtain conditions which ensure that a semicocycle extension $(X_f,G)$ is tame,
we will have to assume a certain form of freeness of its maximal equicontinuous factor $(\T,G)$
(see Lemma~\ref{lem: criterion for tameness}).
This section provides a simple auxiliary statement which will prove useful in this context.

Recall that $G$ is said to act \emph{almost freely} on $\theta\in \T$ if 
there are only finitely many $g\in G$ with $g\theta=\theta$.
In other words, $G$ acts almost freely on $\theta$ if $\Stab_G(\theta)=\{g\in G\: g\theta=\theta\}$ is finite.
To weaken the notion of freeness even further, let us introduce the following relation.
Given $\theta\in \T$ and $g,g'\in\Stab_G(\theta)$, we write
\[
 g\stackrel{\theta}{\sim}g' \quad \Leftrightarrow \quad \text{there is a neighbourhood } U \text{ of } \theta 
 \text{ such that for all } \w \in U \text{ we have } 
 g \w =g'\w.
\]
Clearly, $\stackrel{\theta}{\sim}$ defines an equivalence relation on $\Stab_G(\theta)$.

\begin{defn}
 We say that $G$ acts \emph{locally almost freely} on $\theta\in \T$ if
 the quotient of $\Stab_G(\theta)$ with respect to $\stackrel{\theta}{\sim}$ is finite.
\end{defn}

It is immediate that $G$ acts locally almost freely on $\theta$ if it acts almost freely on $\theta$.
An example of an effective minimal equicontinuous system $(\T,G)$ where $G$ acts locally almost freely on every $\theta$
but not almost freely on any $\theta$ is given by an action of the isometric subgroup of the topological full group
of a $\Z$-odometer (see \cite[Example~7.3]{FuhrmannGrogerLenz2018}).

The following observation will be applied in the proof of Lemma~\ref{lem: criterion for tameness}.
\begin{prop}\label{prop: loacally almost free}
 Let $(X_f,G)$ be a semicocycle extension of the minimal equicontinuous system $(\T,G)$ and let
 $\pi$ be the corresponding factor map from $(X_f,G)$ to $(\T,G)$.
 Let $\theta\in \T$, $x\in \pi^{-1}(\theta)$, and $g,g_1,g_2\in G$ with $g_1\stackrel{\theta}{\sim} g_2$.
 Then
 \[
  x(gg_1)=x(gg_2).
  \]
\end{prop}
\begin{proof}
 We keep the notation as in the proof of Theorem~\ref{thm: equivalence semicocylce extension and almost automorphic system}.
 Let $(h_n)$ be a net in $\T$ such that $x(g)=\lim_n f(g  h_n \theta_0)$ for each $g\in G$.
 Since $g_1$ and $g_2$ coincide on a neighbourhood of $\theta$, we obtain
 \begin{align*}
  x(g g_1)=\lim_n f(g g_1  h_n \theta_0)=\lim_n f(g g_2  h_n \theta_0)=x(g g_2),
 \end{align*}
 where we used $\theta =\lim_n h_n\theta_0$ (see equation \eqref{eq: defn of factor pi}) in the second equality.
\end{proof}

\section{Regular non-tame and tame non-null examples}\label{seq: examples}
Throughout this section, we assume $\T$ to be an infinite compact metric space equipped with a metric $\rho$
with respect to which the group $G$ acts isometrically on $\T$.
As before, we assume $(\T,G)$ to be minimal.

After a short discussion of (weak forms of) topological independence,
we will turn to the construction of tame non-null and non-tame semicocycle extensions of $(\T,G)$.
In the last section, we provide several symbolic examples.

\subsection{Tameness and nullness}
In the following, we briefly 
discuss the concepts of tameness and nullness.
For the sake of a concise presentation and later applications, this discussion is held
in the framework of semicocycle extensions $(X_f,G)$.

Given subsets $A_0,A_1 \ssq X_f$, we say that $J\ssq G$ is an \emph{independence set} for $(A_0,A_1)$ if
for each finite subset $I\ssq J$ and every $a\in \{0,1\}^I$ we have
$\bigcap_{i\in I}\sigma^{i^{-1}}A_{a_i}\neq \emptyset$.
A pair of points $x_0,x_1 \in X_f$ is an \emph{IT-pair} if
for each pair of neighbourhoods $U_0$ and $U_1$ of $x_0$ and $x_1$, respectively,
there is an infinite independence set.
Instead of providing the original definition of tameness, we make use of the following alternative characterisation
(see \cite[Proposition~6.4]{KerrLi2007}).
We say $(X_f,G)$ is \emph{non-tame} if there is an IT-pair $(x_0,x_1)$
with $x_0\neq x_1$.
Note that the existence of such an IT-pair implies the existence
of distinct $k_0,k_1\in K$ such that for all compact neighbourhoods $V_0\ssq K$ and $V_1\ssq K$ of $k_0$ and $k_1$, respectively,
there is an infinite independence set $J$ for $(A(V_0), A(V_1))$,
where
\[
	A(V_j)=\{x\in X_f\: x_e\in V_j\} \qquad (j=0,1).
\]

Vice versa, the existence of disjoint compact subsets $A_0,A_1\in X_f$ with an infinite independence set
implies non-tameness \cite[Proposition~6.4]{KerrLi2007}.
This immediately yields the following
\begin{prop}\label{prop: non-tameness iff}
 The system $(X_f,G)$ is non-tame if and only if there are disjoint compact sets $V_0,\, V_1\ssq K$
 and a sequence $(g_i)_{i\in \N}$ in $G$ such that for every $\ell\in \N$ and each
 $a\in \{0,1\}^\ell$ there is
 $x\in X_f$ with $x_{g_i}\in V_{a_i}$ for $i=1,\ldots,\ell$ or, equivalently, such that for each
 $a\in \{0,1\}^\N$ there is
 $x\in X_f$ with $x_{g_i}\in V_{a_i}$.
\end{prop}

The following observation will be key in the construction of tame examples.

\begin{lem}\label{lem: criterion for tameness}
 Let $(X_f,G)$ be a semicocycle extension of the minimal equicontinuous metric dynamical system $(\T,G)$.
 Suppose $D_f$ is countable and suppose that for each $\theta\in \T$ we have that $G\theta \cap D_f$ is finite and
 that $G$ acts locally almost freely on every $\theta\in D_f$.
  Then $(X_f,G)$ is tame.
\end{lem}
\begin{proof}
 For a contradiction, suppose we are given disjoint compact sets $V_0,\, V_1\ssq K$ and a sequence $(g_i)_{i\in \N}$ in $G$
 as in Proposition~\ref{prop: non-tameness iff}.
 Without loss of generality, we may assume that there is $\xi \in E(\T)$ with
 $g_i\to \xi$ as $i\to \infty$ (where the $g_i$ are considered as elements in $E(\T)$).

 Given $x\in X_f$, recall that $x_{g_i}\in F(g_i\pi(x))$ (see equation \eqref{eq: coordinate entries lie in sections}),
 where $\pi$ denotes the factor map from $(X_f,G)$ to $(\T,G)$. 
 Let us assume first that $\lim_{i\to\infty}g_i \pi(x)=\xi\pi(x)\notin D_f$.
 Note that in this case there is $i_0\in\N$ such that at least one of the following conditions holds
 \begin{itemize}
  \item $\forall\, i\geq i_0\: x_{g_i}\notin V_0$ (which is the case if $F(\xi\pi(x))=\{f(\xi\pi(x))\}\ssq V_0^c$) or
  \item $\forall\, i\geq i_0\: x_{g_i}\notin V_1$ (which is the case if $F(\xi\pi(x))=\{f(\xi\pi(x))\}\ssq V_1^c$),
 \end{itemize}
 where we understand $f$ to be defined on $\T\setminus D_f$ in the sense of \eqref{eq: semicocycle extended to complement of Df}.
 Hence, if there is $a\in\{0,1\}^\N$ with $x_{g_i}\in V_{a_i}$ for every $i\in \N$, then $a$ is
 necessarily eventually constant.
 Thus, the set
 \[
 \left\{a\in \{0,1\}^\N\: \text{there is } x\in X_f \text{ with } \xi\pi(x)\notin D_f \text{ and } x_{g_i}\in V_{a_i} \ (i\in \N) \right\}
 \]
 consists of eventually constant sequences and is therefore at most countable.

 Since $D_f$ is countable, we clearly have that $\{\theta\in \T\: \xi\theta\in D_f\}$ is countable, too.
 As $\{0,1\}^\N$ is uncountable, there must hence be
 $\theta \in \T$ (with $\xi\theta \in D_f$) such that
 \[
  \left\{a\in \{0,1\}^\N\: \text{there is } x\in \pi^{-1}(\theta) \text{ with } x_{g_i}\in V_{a_i} \ (i\in \N)\right\}
 \]
 is uncountable.
 Pick such $\theta \in \T$.
 
 Suppose we are given $i_0\in \N$ with $g_{i_0}\theta\notin D_f$.
 Then for all $x\in \pi^{-1}(\theta)$ we have that $x_{g_{i_0}}\in F(g_{i_0}\theta)=\{f(g_{i_0}\theta)\}$, due to \eqref{eq: coordinate entries lie in sections}.
 In other words, there is only one (if any) admissible value for the ${i_0}$-th entry of
 any sequence $a\in \{0,1\}^\N$ verifying $x_{g_i}\in V_{a_i}$ $(i\in \N)$ for some $x\in \pi^{-1}(\theta)$.
 
 We may hence assume without loss of generality that $g_i\theta \in D_f$ for all $i\in \N$.
 By our assumptions, there are finitely many $\theta_1,\ldots,\theta_n$ 
 such that $G \theta\cap D_f=\{\theta_1,\ldots,\theta_n\}$. 
 We may assume that $g_\ell\theta=\theta_\ell$ for $\ell=1,\ldots,n$.
 
 Clearly, $\Stab_G(\theta)=g_1^{-1}\Stab_G(\theta_1)g_1$ so that $G$ acts locally almost freely on
 $\theta$, too.
 Let $a_1,\ldots,a_m$ be representatives of the equivalence classes of $\stackrel{\theta}{\sim}$.
 For $i\in \N$, we set $\ell(i)\in\{1,\ldots,n\}$ such that $g_{\ell(i)}\theta=g_i\theta$ and set
 $j(i)\in\{1,\ldots,m\}$ such that $g_{\ell(i)}^{-1}g_i\stackrel{\theta}{\sim}a_{j(i)}$.
 Now, for all $i\in \N$ and all $x\in \pi^{-1}(\theta)$, we obtain $x_{g_i}=x_{g_{\ell(i)}g_{\ell(i)}^{-1}g_i}=x_{g_{\ell(i)}a_{j(i)}}$  where we used
 Proposition~\ref{prop: loacally almost free} in the last step. 
 Let $N\in \N$ be such that for all $i\in \N$ there is $k_i\leq N$ with $\ell(k_i)=\ell(i)$ and $j(k_i)=j(i)$.
 Observe that such $N$ clearly exists.
 However, this implies that every sequence $a\in \{0,1\}^\N$, which satisfies for some $x\in \pi^{-1}(\theta)$ and
 all $i\in  \N$ that $x_{g_i}\in V_{a_i}$, is completely determined by its first $N$ entries.
 That is, there are only finitely many such sequences $a$.
 This contradicts the assumptions on $\theta$ and finishes the proof.
\end{proof}

Naturally related to the idea of tameness is the concept of \emph{nullness}.
We refrain from rephrasing the original definition.
Instead, we provide the following characterisation for systems of the form $(X_f,G)$
which is obtained by similar arguments as Proposition~\ref{prop: non-tameness iff} (cf. \cite[Proposition~5.4]{KerrLi2007}).
\begin{prop}\label{prop: non-null iff}
 The system $(X_f,G)$ is non-null if and only if there are disjoint compact sets $V_0,\, V_1\ssq K$ such that for each $\ell\in \N$
 there is a finite sequence $(g_i)_{i=1,\ldots,\ell}$ in $G$ such that for each $a\in \{0,1\}^\ell$ there exists
 $x\in X_f$ with $x_{g_i}\in V_{a_i}$.
\end{prop}

\subsection{Technical preparations}\label{sec: technical prep}
In this part, we provide some tools which will be successively used in the next sections.
We start by defining properties of a family of sequences $(r^n_i)_{i=1}^\infty$ where $n\in \N$ of real numbers
which will serve as radii of certain balls in a later step of our construction. 
In particular, we ask that the following holds for every $n\in \N$ (if applicable)
\begin{enumerate}[label=\textnormal{(R\arabic*)}]
 \item \label{cond:R0} $(r_1^n)_{n\in \N}$ is a strictly decreasing null-sequence.
 \item \label{cond:R1} $(r^n_i)_{i\in\N}$ is a strictly decreasing null-sequence.
 \item \label{cond:R2} There exists $m_n \in 4\N+1$ such that  $r_1^{n+1}=r^n_{m_n}$.
 \item \label{cond:R3} Suppose $r_j^{n+1}=r_i^n$ for some $i,j\in \N$. 
 If $j=1 \mod 4$, then $r_{j+1}^{n+1}=r_{i+1}^n$.
 If $j=2 \mod 4$, then $r_{j+3}^{n+1}=r_{i+1}^n$.
\item \label{cond:R4} For all $\theta\in \T$ and every $i\in \N$,
we have $B_{r_i^n}(\theta)\setminus \overline{B_{r_{i+1}^n}(\theta)}\neq \emptyset$. 
\end{enumerate}
Notice that \ref{cond:R2} and \ref{cond:R3} imply
that the sequence $r^{n+1}$ is obtained from  $(r^n_i)_{i=m_n}^\infty$ 
by adding two extra entries $r_{j+1}^{n+1}$ and $r_{j+2}^{n+1}$ between $r_{i}^n=r_{j}^{n+1}$ 
and $r_{i+1}^n=r_{j+3}^{n+1}$ for each even index $i\ge m_n$ and appropriately chosen $j\in\N$.

Observe that it is always possible to find a family with the above properties:
First, choose a null-sequence $(r_i)$ such that for some $\theta\in \T$ we have
$B_{r_i}(\theta)\setminus \overline{B_{r_{i+1}}(\theta)}\neq \emptyset$.
This is possible since under the present assumptions $\T$ cannot have isolated points.
Now, as pointed out in
Remark~\ref{rem: action of ET on T}, $E(\T)$ acts transitively and isometrically
on $\T$ which clearly gives that $B_{r_i}(\theta)\setminus \overline{B_{r_{i+1}}(\theta)}\neq \emptyset$
actually holds for all $\theta \in \T$.
An appropriate re-labelling of $(r_i)$ (in an obviously non-injective fashion) provides us with a family $(r_i^n)$ which satisfies the above.
It is worth mentioning and will be used frequently that we can choose the radii $r_1^n$
arbitrarily small.

By means of the above radii $(r_i^n)$, we now construct real-valued functions $f_n$ which will 
be the building blocks of the semicocycles in the next section.
Given $n\in \N$, let $f_n\:(0,\infty)\to[0,1]$ be a 
continuous function which vanishes outside $(0,r_1^n)$ and
verifies
$f_n(x)=1$ for all $x\in[r^n_{i+1},r^n_{i}]$ when $i=2\mod 4$ and
$f_n(x)=0$ for all $x\in[r^n_{i+1},r^n_{i}]$ when $i=0\mod 4$.
\begin{figure}[!h]
\begin{center}
 \includegraphics[scale=2]{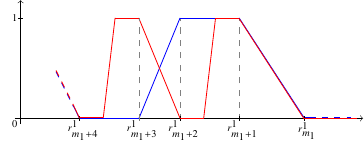}
 \caption{The graphs of $f_1$ (blue) and $f_2$ (red).
 }\label{pic: 1}
\end{center}
\end{figure}
Obviously, $f_n$ cannot be extended to a continuous function on $[0,\infty)$.
The next statement follows from the fact that for each $n\ge 1$ the function 
$f_{n+1}$ assumes both the values $0$ and $1$ on all but finitely many
of those intervals of the form $[r^n_{i+1},r_i^n]$ on which $f_n$ is constant
(specifically: on all such intervals
contained in $(0,r_1^{n+1})$, see Figure~\ref{pic: 1}). 

\begin{lem}\label{lem: free interval of f_n}
 For each $\alpha \in \N_0$, each $s\in\N$ and all $a\in\{0,1\}^s$,
 there is $j\in \N$ such that with $I_a^{\alpha}=[r_{j+1}^{\alpha+s},r_{j}^{\alpha+s}]$  we have
 $(f_n(x))_{n=\alpha+1,\ldots,\alpha+s}=a$ for all $x\in I_a^{\alpha}$.
\end{lem}
\begin{rem}\label{rem: I is rho-realisable}
 Notice that \ref{cond:R4} yields that for each $\theta\in \T$ there is $\theta'\in\T$ with $\rho(\theta,\theta')\in I_a^\alpha$.
\end{rem}

\subsection{Tame non-null extensions}
 We now turn to the construction of tame but non-null semicocycle extensions.
 Note that such extensions ask for at least two orbits in $\T$: one containing $\theta_0$ (along which $f$ is continuous)
 and another one which hits a non-empty set of discontinuity points $D_f$ of $f$.
\begin{thm}\label{thm:tame_non-null_example}
 Let $\T$ be an infinite compact metric space on which $G$ acts minimally by isometries.
 Suppose there are at least two distinct $G$-orbits in $\T$
 and assume there is a point $\theta\in \T$ on which $G$ acts locally almost freely.
 Then there exists a tame and non-null almost automorphic extension
 $(X_{f},G)$ of $(\T,G)$.
 
 If, additionally, for one (and hence every) point $\theta\in \T$ the orbit $G\theta\ssq \T$ is measurable, then
 $(X_{f},G)$ can be chosen to be regular.
\end{thm}
\begin{rem}
Observe that a sufficient condition for orbits in $\T$ to be measurable is to assume that $G$ is $\sigma$-compact:
 Then, each $G$-orbit is $\sigma$-compact as well and hence a 
 countable union of compact sets and therefore measurable.
\end{rem}
\begin{proof}[Proof of Theorem~\ref{thm:tame_non-null_example}]
Pick $\theta\in \G$ such that $G$ acts locally almost freely on $\theta$,
let $g_0$ coincide with the neutral element $e_G$ in $G$
and choose a sequence $(g_n)_{n=1}^\infty$ in $G$ such 
that $(g_n\theta)_{n=0}^\infty$ has pairwise distinct elements and
$g_n\theta\to \theta$ as $n\to\infty$. 

Consider a collection of radii $\{(r^n_i)_{i=1}^\infty :n=1,2,\ldots\}$ 
which satisfies \ref{cond:R0}--\ref{cond:R4}. 
Without loss of generality, we may assume that for each $n\ge 1$ the radius $r^n_1$ is sufficiently small 
to guarantee that
\begin{align}\label{eq: disjoint balls non-null}
 \{B_{r_1^n}(g_n\theta):n=1,2,\ldots\} 
\end{align}
is a collection of pairwise disjoint balls. 

For $n\in \N$, let $f_n\colon(0,\infty)\to[0,1]$ denote the function associated to
the sequence $(r^n_i)_{i=1}^\infty$ as described in Section~\ref{sec: technical prep}. 
Given $s\in \N$, set $\alpha(s)=\sum_{j=0}^{s-1} j$ and let $J_s\ssq(0,\infty)$ be a closed interval such that 
$\bigcup_{a\in \{0,1\}^{s}} I_a^{\alpha(s)}\ssq J_s$ and
$f_{\alpha(s)+1}(x)=f_{\alpha(s)+2}(x)=\ldots=f_{\alpha(s+1)}(x)=0$ if $x$ is a boundary point of $J_s$.
Here, the intervals $I_a^{\alpha(s)}$ are provided by Lemma~\ref{lem: free interval of f_n}.

Given $n\in \N$ with $\alpha(s)< n\leq  \alpha(s+1)$, define $\bar f_n \:[0,\infty)\to [0,1]$ by
\[
\bar{f}_n(x)=\begin{cases}
f_n(x) &\text{if }x \in J_s,\\
0 &\text{otherwise.}
\end{cases}
\]
Note that all $\bar f_n$ are continuous.
Further, for every $s\in \N$, all $\ell=1,\ldots,s$, every $a\in \{0,1\}^s$, and every $x\in I_a^{\alpha(s)}\ssq J_s$
we have
\begin{align}\label{eq: bar f verifies free lemma}
 \bar{f}_{\alpha(s)+\ell}(x)={f}_{\alpha(s)+\ell}(x)=a_\ell,
\end{align}
due to Lemma~\ref{lem: free interval of f_n}.

Now, for $\w\in \T$ define 
\begin{equation*}
{f}(\omega)=\begin{cases}
              0 & \mbox{if } \omega=\theta, \\
              \sum_{n=1}^\infty \bar{f}_n(\rho(\w,g_n\theta)) & \mbox{otherwise}.
            \end{cases}
\end{equation*}
Due to the disjointness of the collection of balls in \eqref{eq: disjoint balls non-null},
this defines a function ${f}\colon\T\to [0,1]$ which is further continuous on $\T\setminus \{\theta\}$. 
Hence, if $\theta_0\in \T$ is such that $G\theta_0\cap \{\theta\}=\emptyset$, 
then $f$ is a semicocycle over $(\T,G,\theta_0)$.

We construct $F=\overline{\gr f}\ssq \T\times [0,1]$ and $X_{f}$ as described in Section~\ref{sec: semi-cocycle extensions}.
Clearly, $D_{f}=\{\theta\}$.
To see that $f$ is invariant under no rotation, pick any distinct $\theta_1,\theta_2\in \T$ and
choose $\xi\in E(\T)$ such that $\xi\theta_1=\theta$.
Then, $\xi\theta_1\in D_{{f}}$ and $\xi\theta_2\in \T\setminus D_{f}$ so that
Proposition~\ref{prop: criterion imnvariance under no rotation} yields that $f$ is invariant under no rotation.
Hence, by Theorem~\ref{thm: equivalence semicocylce extension and almost automorphic system}, $(X_{f}, G)$ is an almost automorphic extension of $(\G, G)$.
Further, by Lemma~\ref{lem: criterion for tameness}, we immediately obtain that $(X_{f},G)$ is tame.

Now, suppose we are given $a\in \{0,1\}^s$ for some $s\in\N$.
Since $G$ acts minimally on $\T$, we can choose $h_a\in G$ such that
$\rho(h_a\theta_0,\theta)$ is in the interval $I^{\alpha(s)}_{a}$
(see also Remark~\ref{rem: I is rho-realisable}).
As $G$ acts by isometries, we further have $\rho(g_nh_a\theta_0,g_n\theta)=\rho(h_a\theta_0,\theta)$ for every $n\ge 1$.
In particular, this gives
${f}(g_nh_a\theta_0) = {\bar f_n}(\rho(g_nh_a\theta_0,g_n\theta))$
for $n=\alpha(s)+1,\ldots,\alpha(s+1)$, due to
the definition of $f$ and the disjointness of the balls
in \eqref{eq: disjoint balls non-null}.

Hence, with $(x_g)_{g\in G}=({f}(g h_a\theta_0))_{g\in G}$ we have
\[
 x_{g_n} = {\bar f_n}(\rho(g_nh_a\theta_0,g_n\theta))
 ={f_n}(\rho(h_a\theta_0,\theta))=a_n \in V_{a_n} \quad \text{for } n=\alpha(s)+1,\ldots,\alpha(s+1),
\]
where $V_0=\{0\}$ and $V_1=\{1\}$ and where we used \eqref{eq: bar f verifies free lemma}.
Since $s\in\N$ and $a\in \{0,1\}^s$ were arbitrary,
we obtain that $(X_{f},G)$ is non-null
by means of
Proposition~\ref{prop: non-null iff}.

In order to see the second part, first note that the assumption of a measurable orbit implies that
every orbit of $(\T,G)$ is measurable and further, that every orbit is necessarily of 
the same $m_\T$-measure, due to Remark~\ref{rem: action of ET on T}.
Since $m_\T$ is further ergodic and orbits are clearly invariant sets, this implies
that orbits are of measure zero as we assume that there is more than one orbit in $\T$.\footnote{Note that this
immediately yields that there are actually uncountably many orbits in $\T$.}
Now, let $\pi$ denote the factor map from $(X_f,G)$ to $(\T,G)$.
Clearly, the projection of the almost automorphic points $\pi(X_0)$ (see \eqref{eq: defn almost automorphic points})
coincides with the complement of $G\theta$ and is hence of full measure.
\end{proof}

\begin{rem}
 In view of the above theorem, we may ask if there is a minimal isometric dynamical system with more than one orbit such that its almost automorphic extensions are null if and only if they are tame.
 Note that the assumption of a point $\theta\in \T$ on which $G$ acts locally almost freely is only sufficient but not necessary in order to rule out this possibility:
 Clearly, the canonical action of the special orthogonal group $\SO$ on the $2$-sphere $\S^2$ is minimal, isometric
 and not locally almost free on any $\theta \in \S^2$.
 While $(\S^2,\SO)$ only allows for exactly one orbit, we may
 consider the natural action of the product $G=\SO\times H$ on $\T=\S^2\times \T_1$,
 where $(\T_1,H)$ is some minimal isometric dynamical system which allows for at least two distinct orbits.
 It is straightforward to see that $(\T,G)$ is still minimal and isometric and that $(\T,G)$ allows for
 more than one orbit.
 Furthermore, $G$ does not act locally almost freely on any $\theta\in \T$.
 However, we clearly have that if there is some non-null and tame almost automorphic extension $(X,H)$ of $(\T_1,H)$,
 then $(\S^2\times X,G)$ also is a non-null and tame almost automorphic extension of $(\T,G)$.
\end{rem}

\begin{rem}
 According to \cite[Corollary 5.4]{Glasner2018} and
 \cite[Theorem~1.2]{FuhrmannGlasnerJagerOertel2018}, a minimal tame dynamical system on a metric space
 which allows for an invariant measure is necessarily regularly almost automorphic.
 In fact, a close inspection of the proof in \cite{FuhrmannGlasnerJagerOertel2018} shows that
 every tame almost automorphic system with the property that its maximal equicontinuous factor is metrizable
 and such that orbits in $\T$ are measurable is automatically regular.
 Against this background, we could also reformulate the second part of Theorem~\ref{thm:tame_non-null_example}
 as follows:
 If, additionally, for one (and hence every) point $\theta\in \T$ the orbit $G\theta\ssq \T$ is measurable, then
 $(X_{f},G)$ \emph{is} regular.
\end{rem}

\begin{rem}
 It is worth remarking that free minimal equicontinuous systems $(\T,G)$, where $\T$ is assumed to 
 be metrizable and infinite,
 may, in fact, allow for only two orbits: Let $\T$ be a compact topological group which allows for a dense 
 subgroup $G\leq \T$ of index $2$. 
 Then the natural action of $G$ on $\T$ has exactly two orbits and is clearly free and minimal.
 
 Note that a well-known example of such $\T$ is given by the product
 $\T=\{0,1\}^\N$ of countably many copies of the finite field $\{0,1\}$.
 We may consider $\T$ as a vector space (and hence a group) over 
 the field $\{0,1\}$. 
 Pick a base $\mc B$ of $\T$ which contains in particular those elements which have exactly one entry equal to $1$.
 Observe that $\mc B$ necessarily contains an element $b$ with infinitely many
 entries equal to $0$ and infinitely many entries equal to $1$.
 Now, it is straightforward to see that the linear span of $\mc B\setminus \{b\}$ is a subgroup of
 index $2$ in $\{0,1\}^\N$ while it is clearly dense in $\T$.
\end{rem}

\subsection{Non-tame extensions}
We now turn to the construction of non-tame examples.
As Lemma~\ref{lem: criterion for tameness} suggests, such examples may in general need at least 
countably many discontinuity points 
of the associated semicocycle.
Therefore, the invariance under no rotation has to be obtained in a more elaborate fashion than in the previous section.
\begin{thm}\label{thm: non-tame and non-null examples}
 Let $\T$ be an infinite compact metric space on which $G$ acts minimally by isometries.
 Suppose there are at least two distinct orbits in $\T$ under $G$.
 Then there exists a non-tame almost automorphic extension
 $(X_{f},G)$ of $(\T,G)$.
 
 If, additionally, for one (and hence every) point $\theta\in \T$ the orbit $G\theta\ssq \T$ is measurable, then
 $(X_{f},G)$ can be chosen to be regular.
\end{thm}
\begin{proof}
Pick distinct $\theta,\theta'\in \G$ from one and the same orbit, that is, $G\theta= G\theta'$.
Let $r>0$ be such that $\rho(\theta,\theta')>2r$.
Due to Corollary~\ref{cor: stabilizers have empty interior}, there is $g_1\in G$ with
$B_r(\theta)\ni g_1\theta\neq \theta$ and $g_1\theta'\neq \theta'$.
Let $g_0$ coincide with the neutral element $e_G$ of $G$ and choose a sequence $(g_n)_{n=2}^\infty$ in $G$ such that
$(g_n \theta)_{n=0}^\infty$ consists of pairwise distinct elements with $g_n\theta\to \theta$ as $n\to\infty$
and $g_n\theta \in B_r(\theta)\cap B_r(g_1\theta)$ for all $n\in \N$. 

Consider a collection of radii $\{(r^n_i)_{i=1}^\infty :n=1,2,\ldots\}$ 
which satisfies \ref{cond:R0}--\ref{cond:R4}. 
By choosing the radii $r_1^n$ sufficiently small,
we may assume that first,
\begin{align}\label{eq: disjoint balls}
    \{B_{2r_1^1}(g_1\theta)\} \cup \{B_{2r_1^{n-1}}(g_n\theta)\:n=2,3,\ldots\}
\end{align}
is a family of pairwise disjoint balls, secondly,
\begin{align}\label{eq: big ball}
B_r(\theta)\cap B_r(g_1\theta)\supseteq \operatorname{cl}\left({\bigcup_{n=1}^\infty B_{r_1^n}(g_n\theta)}\right)
\end{align}
and thirdly, that
\begin{align}\label{eq: g1 not in stabiliser of theta prime}
 \rho(\theta',g_1\theta')>2\cdot r_1^1.
\end{align}

Let $\Theta=\{g_n\theta\:n= 0,1,\ldots\}\cup\{\theta'\}$. 
Define ${f}\colon\T\to [0,1]$ by
\begin{align}\label{eq: defn non-tame semicocycle}
f(\omega)=\begin{cases}
              0 & \mbox{if } \omega\in\Theta, \\
              f_1(\rho(\w,\theta'))+\sum_{n=1}^\infty f_n(\rho(\w,g_n\theta)) & \mbox{otherwise}.
            \end{cases}
\end{align}
Observe that $f$ is continuous outside the set $\Theta \ssq G\theta$.
Further, by the assumptions, there is $\theta_0\in \T\setminus G\theta$ so that the restriction of $f$ to $G\theta_0$
is, in fact, continuous.
We may hence consider $f$ to be a semicocycle over $(\T,G,\theta_0)$. 
We construct $F$ and $X_{f}$ as described in Section~\ref{sec: semi-cocycle extensions}.
Clearly, $D_{f}=\Theta$.

We next show that $f$ is invariant under no rotation.
For the sake of the construction of symbolic examples in the next section,
we are going to prove slightly more than we actually need for the present purpose.\footnote{In fact, 
in order to immediately obtain invariance under no rotation in 
a way as simple as in Theorem~\ref{thm:tame_non-null_example}, we could construct $f$ in such a way that 
$\theta'$ is the unique point in $D_f$ with $2\in F(\theta)$ (by simply replacing the summand
$f_1(\rho(\omega,\theta'))$ by $2f_1(\rho(\omega,\theta'))$ in \eqref{eq: defn non-tame semicocycle}).
However, as we also aim at symbolic examples, we won't follow this path.}
To that end, let us define 
\begin{align}\label{eq: defn U Theta}
 U_{\Theta}=\operatorname{cl}\left({\bigcup_{n=1}^\infty B_{r_1^{n}}(g_n\theta) \cup B_{r_1^{1}}(\theta')}\right)=
 \operatorname{cl}\left(\bigcup_{n=1}^\infty B_{r_1^{n}}(g_n\theta)\right) \cup \overline{B_{r_1^{1}}(\theta')}.
\end{align}
\begin{claim}\label{claim: invariance by topological means}
 For distinct $\theta_1$ and $\theta_2$ in $\T$, there is 
 $\xi \in E(\T)$ such that $\xi \theta_1 \in \Theta$ and $\xi\theta_2\notin {U_{\Theta}}$.
\end{claim}
Observe that $\Theta \ssq U_{\Theta}$.
Hence, taken the above claim for granted,
we immediately obtain the invariance under no rotation from Proposition~\ref{prop: criterion imnvariance under no rotation}
so that $(X_f,G)$ is indeed a semicocycle extension --and thus, an almost automorphic extension-- of $(\T,G)$.
\begin{proof}[Proof of Claim~\ref{claim: invariance by topological means}]
Fix $\theta_1,\theta_2\in \T$ with $\theta_1\neq \theta_2$.
We have to distinguish between the following cases where we repeatedly use 
that $E(\T)$ acts transitively and isometrically on $\T$ (see Remark~\ref{rem: action of ET on T}).

{\em Case 1 ($0<\rho(\theta_1,\theta_2)\leq r_1^1$):}
In this case, since $(r_1^n)_{n\in \N}$ is a strictly decreasing null sequence (due to \ref{cond:R0}), 
there is $n_0$ such that $r^{n_0}_1\geq \rho(\theta_1,\theta_2)> r_1^{n_0+1}$.
Choose $\xi \in E(\T)$ such that $\xi \theta_1= g_{n_0+1} \theta\in \Theta$.
Then, we have $\xi\theta_2\in \overline{B_{r^{n_0}_1}(g_{n_0+1} \theta)}\setminus \overline {B_{r^{n_0+1}_1}(g_{n_0+1} \theta)}$
so that the disjointness of the family of balls in \eqref{eq: disjoint balls} gives that $\xi\theta_2\notin U_{\Theta}$.

{\em Case 2 ($r_1^1<\rho(\theta_1,\theta_2)\leq r$):}
In this case, choose $\xi \in E(\T)$ such that $\xi \theta_1=\theta'\in \Theta$.
Then $\xi\theta_2\in \overline{B_r(\theta')}\setminus \overline{B_{r_1^1}(\theta')}$
which is clearly in the complement of $U_{\Theta}$ due to \eqref{eq: big ball} and the fact that
$\rho(\theta',\theta)>2r$.

{\em Case 3 ($r<\rho(\theta_1,\theta_2)$):}
Choose $\xi\in E(\T)$ such that $\xi\theta_1=\theta \in \Theta$.
If $\xi\theta_2\notin U_{\Theta}$, we are done.
Hence, it remains to consider $\xi \theta_2 \in U_{\Theta}$.
In this case, we necessarily have $\xi\theta_2 \in \overline{B_{r_1^1}(\theta')}$, due to \eqref{eq: big ball}.
Now, observe that $g_1\xi\theta_1=g_1\theta\in \Theta$.
However, by the reverse triangle inequality, 
\[
 \rho(\theta',g_1\xi\theta_2)\geq\rho(\theta',g_1\theta')-\rho(g_1\theta',g_1\xi\theta_2)=
\rho(\theta',g_1\theta')-\rho(\theta',\xi\theta_2)> 2r_1^1-r_1^1\geq r_1^1, 
\]
where we used \eqref{eq: g1 not in stabiliser of theta prime} in the second to the last step.
Hence, $g_1\xi\theta_2\notin \overline{B_{r_1^1}(\theta')}$.
At the same time, $\rho(g_1\xi\theta_2,g_1\xi\theta_1)=\rho(\theta_2,\theta_1)>r$, so that by \eqref{eq: big ball}
we indeed obtain $g_1\xi\theta_2\notin U_{\Theta}$.
This proves the claim.
\end{proof}

In order to finish the proof of the first part, it remains to show that $(X_{f},G)$
is non-tame.
To that end, suppose we are given $a\in \{0,1\}^s$ for some $s\in\N$.
Choose $h_a\in G$ such that $\rho(h_a\theta_0,\theta)$ is in the interval $I^0_{a}$
from Lemma~\ref{lem: free interval of f_n} which is possible due to Remark~\ref{rem: I is rho-realisable}
and since $(\T,G)$ is minimal.
Since $G$ acts by isometries, we have $\rho(g_nh_a\theta_0,g_n\theta)=\rho(h_a\theta_0,\theta)$
for every $n\ge 1$ 
so that ${f}(g_nh_a\theta_0) =f_n(\rho(h_a\theta_0,\theta))$ for all $n=1,\ldots,s$ 
(due to the definition of $f$ and due the disjointness of the family of balls
in \eqref{eq: disjoint balls}).
Hence, with $(x_g)_{g\in G}=({f}(g h_a\theta_0))_{g\in G}$,
Lemma~\ref{lem: free interval of f_n} gives 
\[
 x_{g_n} =f_n(\rho(h_a\theta_0,\theta))=a_n \in V_{a_n} \quad \text{ for each } n=1,\ldots,s,
\]
where $V_0=\{0\}$ and $V_1=\{1\}$.
Since $s\in\N$ and $a\in \{0,1\}^s$ were arbitrary,
we obtain that $(X_{f},G)$ is non-tame
by means of Proposition~\ref{prop: non-tameness iff}.
This finishes the proof of the first part.

The second part follows similarly as in the proof of Theorem~\ref{thm:tame_non-null_example} since the 
complement of the projection of the almost automorphic points of $X_f$ coincides with $GD_f$ and is hence 
a countable union of orbits.
\end{proof}

\subsection{Symbolic examples} 
Note that in the constructions above we obtained $K$-valued semicocycles with $K=[0,1]$. 
It is natural to ask whether we can find \emph{symbolic} examples, that is, $\{0,1\}$-valued 
semicocycles which yield tame non-null and non-tame extensions, respectively. 
In the proofs above, it was not only important to control the number of times 
a given orbit hits the set of discontinuity points $D_f$ 
but also that there are orbits which don't hit $D_f$ at all.
Due to the fact that we dealt with $[0,1]$-valued examples, we could ensure that the set $D_f$ was at most countable
which simplified the related problems considerably.

Now, in order to restrict the set of values to $\{0,1\}$, we may change the functions $f_n$
to obtain maps $f_n'\:(0,\infty)\to \{0,1\}$ for $n=1,2,\ldots$ by setting
$f_n'(x)=1$ for all $x\in[r^n_i,r^n_{i+1}]$ when $i=2\mod 4$
and $f_n'(x)=0$ otherwise.
Observe that we immediately obtain a similar statement as Lemma~\ref{lem: free interval of f_n} if we replace 
the functions $f_n$ by $f_n'$.
However, when we proceed with these $\{0,1\}$-valued functions as in the proofs of Theorem~\ref{thm:tame_non-null_example}
and Theorem~\ref{thm: non-tame and non-null examples},
we may create new discontinuity points.
This may even imply $GD_f=\T$ which is an obvious obstruction for the constructions.

Nonetheless, under suitable extra assumptions, we can still obtain regular symbolic semicocycle extensions.
On the side of non-tame extensions, we obtain the following statement and en passant a generalisation of
\cite[Theorem~3.1]{Paul1976}.

\begin{cor}\label{cor: symbolic non-tame extensions}
Suppose $G$ is countable and $(\T,G)$ 
is a metric minimal isometric dynamical system (with $\T$ infinite).
Then there is a symbolic regular almost automorphic extension $(X_{f},G)$ which is non-tame.
\end{cor}
\begin{proof}
The proof works almost literally as the proof of Theorem~\ref{thm: non-tame and non-null examples} if we
construct $f$ by means of the functions $f_n'$ instead of $f_n$ and if we assume that the family of radii $(r_i^n)$ verifies
not only \ref{cond:R0}--\ref{cond:R4} but also
\begin{enumerate}[label=\textnormal{(R\arabic*)}, resume]
\item \label{cond:R5} For every $\theta\in \T$ and all $i,n\in \N$
we have $m_{\T}(\overline{B_{r_i^n}(\theta)}\setminus B_{r_i^n}(\theta))=0$.
\end{enumerate}
Note that \ref{cond:R5} holds as soon as there is just one $\theta\in \T$ with
$m_{\T}(\overline{B_{r_i^n}(\theta)}\setminus B_{r_i^n}(\theta))=0$ 
(see also Remark~\ref{rem: action of ET on T}).
This can clearly be realised since $m_\T$ is finite and the family $(r_i^n)$ is countable.

We leave the remaining details of the proof to the reader but would like to make the following comments:
\begin{itemize}
 \item Now, the set of discontinuity points $D_{f}$ is the union of the countable set $\Theta$ and countably many 
 sets of the form $\overline{B_{r_i^n}(\theta)}\setminus B_{r_i^n}(\theta)$ which are assumed to be of $m_\T$-measure
 zero due to \ref{cond:R5}.
 Hence, as $G$ is countable, $m_{\T}(G D_{f})=0$ so that there is a point $\theta_0$ along whose 
 orbit $f$ is indeed continuous.
 \item Since $\Theta\ssq D_{f}\ssq U_\Theta$ (see equation \eqref{eq: defn U Theta}), we obtain that $f$ is indeed 
 invariant under no rotation due to Claim~\ref{claim: invariance by topological means} and 
 Proposition~\ref{prop: criterion imnvariance under no rotation}.
 Hence, $(X_{f},G)$ is an almost automorphic extension of $(\T,G)$ which is
 further regular since $m_{\T}(G D_{f})=0$. \qedhere
\end{itemize}
\end{proof}
\begin{rem}
 For the special case of irrational rotations on $\T=\R/\Z$, a similar result was announced in
\cite[Remark~5.8]{Glasner2018}.
 As a matter of fact, it has been proven in \cite[Corollary~3.7]{FuhrmannGlasnerJagerOertel2018} that a
 symbolic extension $(X_f,\Z)$ of an irrational rotation on $\R/\Z$ is non-tame if $D_f$ is a Cantor set.
 In the light of the above statement and the fact that already a single discontinuity point can destroy nullness
 (according to Theorem~\ref{thm:tame_non-null_example}),
 this seems to suggest that the question of whether an almost automorphic system is tame or null 
 simply boils down to a question of the \emph{size} of $D_f$.
 However, the mechanism which establishes the non-tameness in \cite[Corollary~3.7]{FuhrmannGlasnerJagerOertel2018}
 can easily be destroyed by placing countably many points in the gaps of the respective
 Cantor set.
 That is, for every prescribed value $\gamma$ between $0$ and $1$, there is a $\{0,1\}$-valued semicocycle $f$
 such that $D_f$ is of Hausdorff measure $\gamma$ while $(X_f,\Z)$ is
 a symbolic tame extension of an irrational rotation.
\end{rem}

To obtain tame non-null examples, we have to introduce a stronger version of the assumption \ref{cond:R5} in 
order to still be able to apply Lemma~\ref{lem: criterion for tameness}.
This boils down to restrictions on the space $\T$
which, on the other, also allow for symbolic non-tame extensions even if $G$ is uncountable.

\begin{cor}\label{cor: symbolic non-null extensions}
If $\T$ is a Cantor set and 
$(\T,G)$ is a minimal equicontinuous dynamical system with at least two distinct orbits,
then there is a symbolic almost one-to-one extension $(X_{f},G)$ which is non-tame.

If, additionally, for one (and hence every) point $\theta\in \T$ the orbit $G\theta\ssq \T$ is measurable, then
$(X_{f},G)$ can be chosen to be regular.

If $G$ acts locally almost freely on some point $\theta\in \T$,
then all of the above holds true if we replace \emph{non-tame} by \emph{tame non-null}.
\end{cor}
\begin{proof}
Let us assume to be given a 
family of radii $(r_i^n)$ which verifies
not only \ref{cond:R0}--\ref{cond:R4} but also
\begin{enumerate}[label=\textnormal{(R6')}]
\item \label{cond:R5 prime} 
For every $\theta\in \T$ and all $i,n\in \N$ the ball $B_{r^n_i}(\theta)$ is clopen. 
\end{enumerate}
Under the assumption of \ref{cond:R5 prime}, functions of the form $\omega\mapsto f_n'(\rho(\omega,\theta))$ are continuous
on $\T\setminus \{\theta\}$ so that the proofs of
Theorem~\ref{thm:tame_non-null_example} and
Theorem~\ref{thm: non-tame and non-null examples}, respectively, translate literally to the present setting.

To see that \ref{cond:R5 prime} can always be ensured under the above hypothesis,
note that we may assume without loss of generality that $\T$ is equipped
with the compatible $G$-invariant metric $\rho$ given by
\[
 \rho(\theta,\theta')=\sup_{g\in G} d(h(g\theta),h(g\theta')),
\]
where $h$ denotes a homeomorphism from $\T$ to $\{0,1\}^\N$
and $d$ denotes the Cantor metric $d(x,y)=2^{-\min\{n\: x_n\neq y_n\}}$ on $\{0,1\}^\N$
which only assumes values in $\{0\}\cup\{1/2^\ell\: \ell \in \N \}$.
Hence, we can assume the metric $\rho$ to assume only countably many values, too.
This certainly allows to guarantee 
that for some $\theta\in\T$ and all $i,n\in\N$, the balls $B_{r^n_i}(\theta)$ are clopen.
As in the previous examples, Remark~\ref{rem: action of ET on T} yields that this carries over to all $\theta\in \T$ so that
\ref{cond:R5 prime} can be realised.
\end{proof}

If $\T$ is a Cantor set and $G$ a discrete countable group, then
the family of free minimal equicontinuous systems $(\T,G)$ is well understood:
the group $G$ is necessarily residually finite and $(\T,G)$ is isomorphic
to a $G$-odometer (see \cite[Theorem 2.7]{CortezMedynets2016}). 
The next statement follows from \cite[Theorem 2.7]{CortezMedynets2016} and
Corollary~\ref{cor: symbolic non-tame extensions} and Corollary~\ref{cor: symbolic non-null extensions}
combined with a characterisation of \emph{Toeplitz flows} as symbolic almost one-to-one extensions
of free minimal $G$-odometers.
For a thorough discussion of $G$-odometers and Toeplitz flows over residually finite groups, we refer
the reader to \cite{Krieger2010,CortezPetite2014,CortezMedynets2016}. 

\begin{cor}\label{cor: Toepltz non-tame non-null extensions}
Let $G$ be a countable discrete group and assume $\T$ is a Cantor set.
Then every free minimal equicontinuous $G$-action $(\T,G)$ is the maximal equicontinuous 
factor of regular Toeplitz flows $(X_1,G)$ and $(X_2,G)$ such that $(X_1,G)$ is non-tame and $(X_2,G)$ is
tame but non-null.
\end{cor}

The point in the proof of Corollary~\ref{cor: symbolic non-null extensions}
is that despite the fact that $f$ only assumes values in $\{0,1\}$,
the set of discontinuity points is still only countable or even finite.
While in general, such a straightforward argument is not available, we obtain
\begin{cor}
	Every irrational rotation on $\R/\Z$ allows for a symbolic almost one-to-one extension
	which is tame but non-null.
\end{cor}
\begin{proof}
As in the previous examples, we leave the details to the reader and 
only briefly discuss the differences to the proofs in the previous section.
Suppose we have a rotation by an irrational angle $\alpha$.
We proceed similarly as in the proof of Theorem~\ref{thm:tame_non-null_example} where
we replace the functions $f_n$ by $f_n'$.
This time, we choose the family $(r_i^n)$ to satisfy \ref{cond:R0}--\ref{cond:R4} and
further assume that for distinct $r_i^n\neq r_j^m$ we have that
$\Z \alpha+ r_i^n \cap \Z\alpha + r_j^m=\emptyset$.
As $r_1^n\to 0$ (due to \ref{cond:R0}), this ensures that every orbit hits the
countable set of the respective discontinuity points
at most finitely many times which again allows the application
of Lemma~\ref{lem: criterion for tameness}.
\end{proof}

\bibliography{tameness-in-semi-cocycles}{}
\bibliographystyle{unsrt}
\end{document}

%% file: bibstandard.tex
\newcommand{\ssq}{\ensuremath{\subseteq}}

\newcommand{\alphlist}{\begin{list}{(\alph{enumi})}{\usecounter{enumi}\setlength{\parsep}{2pt}
      \setlength{\itemsep}{1pt} \setlength{\topsep}{5pt}
      \setlength{\partopsep}{3pt}}}
\newcommand{\arablist}{\begin{list}{(\arabic{enumi})}{\usecounter{enumi}\setlength{\parsep}{2pt}
          \setlength{\itemsep}{1pt} \setlength{\topsep}{5pt}
          \setlength{\partopsep}{3pt}}}
\newcommand{\romanlist}{\begin{list}{(\roman{enumi})}{\usecounter{enumi}\setlength{\parsep}{2pt}
              \setlength{\itemsep}{1pt} \setlength{\topsep}{5pt}
              \setlength{\partopsep}{3pt}}}
\newcommand{\Romanlist}{\begin{list}{(\Roman{enumi})}{\usecounter{enumi}\setlength{\parsep}{2pt}
              \setlength{\itemsep}{1pt} \setlength{\topsep}{5pt}
              \setlength{\partopsep}{3pt}}}
\newcommand{\bulletlist}{\begin{list}{$\bullet$}{\setlength{\parsep}{2pt}
                \setlength{\itemsep}{1pt} \setlength{\topsep}{5pt}
                \setlength{\partopsep}{3pt}\setlength{\leftmargin}{15pt}}}
\newcommand{\Alphlist}{\begin{list}{(\Alph{enumi})}{\usecounter{enumi}\setlength{\parsep}{2pt}
      \setlength{\itemsep}{1pt} \setlength{\topsep}{5pt}
      \setlength{\partopsep}{3pt}}}
 \newcommand{\listend}{\end{list}}

\newcommand{\T}{\ensuremath{\mathbb{T}}}

\newcommand{\N}{\ensuremath{\mathbb{N}}}
\newcommand{\R}{\ensuremath{\mathbb{R}}}
\newcommand{\Z}{\ensuremath{\mathbb{Z}}}

%% file: tameness-general-groups_Fuhrmann_Kwietniak.bbl
\begin{thebibliography}{10}

\bibitem{Blanchard1993}
F.~Blanchard.
\newblock A disjointness theorem involving topological entropy.
\newblock {\em Bull. Soc. Math. France}, 121(4):465--478, 1993.

\bibitem{HLY2012}
Wen Huang, Hanfeng Li, and Xiangdong Ye.
\newblock Family independence for topological and measurable dynamics.
\newblock {\em Trans. Amer. Math. Soc.}, 364(10):5209--5242, 2012.

\bibitem{KerrLi2007}
D.~Kerr and H.~Li.
\newblock Independence in topological and {$C^*$}-dynamics.
\newblock {\em Math. Ann.}, 338(4):869--926, 2007.

\bibitem{KerrLiBook}
D.~Kerr and H.~Li.
\newblock {\em Ergodic theory}.
\newblock Springer Monographs in Mathematics. Springer, Cham, 2016.
\newblock Independence and dichotomies.

\bibitem{AdlerKonheimMcAndrew1965}
R.~L. Adler, A.~G. Konheim, and M.~H. McAndrew.
\newblock Topological entropy.
\newblock {\em Trans. Amer. Math. Soc.}, 114:309--319, 1965.

\bibitem{Goodman1974}
T.~N.~T. Goodman.
\newblock Topological sequence entropy.
\newblock {\em Proc. London Math. Soc. (3)}, 29:331--350, 1974.

\bibitem{Kohler1995}
A.~K\"{o}hler.
\newblock Enveloping semigroups for flows.
\newblock {\em Proc. Roy. Irish Acad. Sect. A}, 95(2):179--191, 1995.

\bibitem{Glasner2006}
E.~Glasner.
\newblock On tame dynamical systems.
\newblock {\em Colloq. Math.}, 105(2):283--295, 2006.

\bibitem{Glasner2007}
E.~Glasner.
\newblock The structure of tame minimal dynamical systems.
\newblock {\em Ergodic Theory Dynam. Systems}, 27(6):1819--1837, 2007.

\bibitem{GlasnerMegrelishvili2006}
E.~Glasner and M.~Megrelishvili.
\newblock Hereditarily non-sensitive dynamical systems and linear
  representations.
\newblock {\em Colloq. Math.}, 104(2):223--283, 2006.

\bibitem{Huang2006}
W.~Huang.
\newblock Tame systems and scrambled pairs under an abelian group action.
\newblock {\em Ergodic Theory Dynam. Systems}, 26(5):1549--1567, 2006.

\bibitem{Romanov2016}
A.~V. Romanov.
\newblock Ergodic properties of discrete dynamical systems and enveloping
  semigroups.
\newblock {\em Ergodic Theory Dynam. Systems}, 36(1):198--214, 2016.

\bibitem{GlasnerMegreshvili2018}
E.~Glasner and M.~Megrelishvili.
\newblock More on tame dynamical systems.
\newblock In {\em Ergodic theory and dynamical systems in their interactions
  with arithmetics and combinatorics}, volume 2213 of {\em Lecture Notes in
  Math.}, pages 351--392. Springer, Cham, 2018.

\bibitem{GlasnerMegrelishvili2012}
E.~Glasner and M.~Megrelishvili.
\newblock Representations of dynamical systems on {B}anach spaces not
  containing {$l_1$}.
\newblock {\em Trans. Amer. Math. Soc.}, 364(12):6395--6424, 2012.

\bibitem{GlasnerMegrelishvili2018Monatshefte}
E.~Glasner and M.~Megrelishvili.
\newblock Circularly ordered dynamical systems.
\newblock {\em Monatsh. Math.}, 185(3):415--441, 2018.

\bibitem{Aujogue2015}
J.-B. Aujogue.
\newblock Ellis enveloping semigroup for almost canonical model sets of an
  {E}uclidean space.
\newblock {\em Algebr. Geom. Topol.}, 15(4):2195--2237, 2015.

\bibitem{ChernikovSimon2018}
A.~Chernikov and P.~Simon.
\newblock Definably amenable {NIP} groups.
\newblock {\em J. Amer. Math. Soc.}, 31(3):609--641, 2018.

\bibitem{GlasnerMegreshvili2013}
E.~Glasner and M.~Megrelishvili.
\newblock Banach representations and affine compactifications of dynamical
  systems.
\newblock In {\em Asymptotic geometric analysis}, volume~68 of {\em Fields
  Inst. Commun.}, pages 75--144. Springer, New York, 2013.

\bibitem{Ibarlucia2016}
T.~Ibarluc\'{i}a.
\newblock The dynamical hierarchy for {R}oelcke precompact {P}olish groups.
\newblock {\em Israel J. Math.}, 215(2):965--1009, 2016.

\bibitem{Glasner2018}
E.~Glasner.
\newblock The structure of tame minimal dynamical systems for general groups.
\newblock {\em Invent. Math.}, 211(1):213--244, 2018.

\bibitem{FuhrmannGlasnerJagerOertel2018}
Gabriel {Fuhrmann}, Eli {Glasner}, Tobias {J{\"a}ger}, and Christian {Oertel}.
\newblock {Irregular model sets and tame dynamics}.
\newblock {\em arXiv e-prints}, pages 1--22, arXiv:1811.06283, Nov 2018.

\bibitem{DownarowiczDurand2002}
T.~Downarowicz and F.~Durand.
\newblock Factors of {T}oeplitz flows and other almost {$1-1$} extensions over
  group rotations.
\newblock {\em Math. Scand.}, 90(1):57--72, 2002.

\bibitem{Paul1976}
M.E. Paul.
\newblock Construction of almost automorphic symbolic minimal flows.
\newblock {\em General Topology and Appl.}, 6(1):45--56, 1976.

\bibitem{Auslander1988}
J.~Auslander.
\newblock {\em Minimal flows and their extensions}, volume 153 of {\em
  North-Holland Mathematics Studies}.
\newblock North-Holland Publishing Co., Amsterdam, 1988.
\newblock Notas de Matem\'atica [Mathematical Notes], 122.

\bibitem{Glasner2003}
Eli Glasner.
\newblock {\em Ergodic theory via joinings}, volume 101 of {\em Mathematical
  Surveys and Monographs}.
\newblock American Mathematical Society, Providence, RI, 2003.

\bibitem{deVries1993}
J.~de~Vries.
\newblock {\em Elements of topological dynamics}, volume 257 of {\em
  Mathematics and its Applications}.
\newblock Kluwer Academic Publishers Group, Dordrecht, 1993.

\bibitem{Downarowicz2005}
T.~Downarowicz.
\newblock Survey of odometers and {T}oeplitz flows.
\newblock In {\em Algebraic and topological dynamics}, volume 385 of {\em
  Contemp. Math.}, pages 7--37. Amer. Math. Soc., Providence, RI, 2005.

\bibitem{DownarowiczSerafin2005}
Tomasz Downarowicz and Jacek Serafin.
\newblock Semicocycle extensions and the stroboscopic property.
\newblock {\em Topology Appl.}, 153(1):97--106, 2005.

\bibitem{MarkleyPaul1976}
Nelson~G. Markley and Michael~E. Paul.
\newblock Almost automorphic symbolic minimal sets without unique ergodicity.
\newblock {\em Israel J. Math.}, 34(3):259--272 (1980), 1979.

\bibitem{Stone1956}
A.~H. Stone.
\newblock Metrizability of decomposition spaces.
\newblock {\em Proc. Amer. Math. Soc.}, 7:690--700, 1956.

\bibitem{FuhrmannGrogerLenz2018}
Gabriel {Fuhrmann}, Maik {Gr{\"o}ger}, and Daniel {Lenz}.
\newblock {The structure of mean equicontinuous group actions}.
\newblock {\em arXiv e-prints}, pages 1--33, arXiv:1812.10219, Dec 2018.

\bibitem{CortezMedynets2016}
Mar\'{i}a~Isabel Cortez and Konstantin Medynets.
\newblock Orbit equivalence rigidity of equicontinuous systems.
\newblock {\em J. Lond. Math. Soc. (2)}, 94(2):545--556, 2016.

\bibitem{Krieger2010}
Fabrice Krieger.
\newblock Toeplitz subshifts and odometers for residually finite groups.
\newblock In {\em \'{E}cole de {T}h\'{e}orie {E}rgodique}, volume~20 of {\em
  S\'{e}min. Congr.}, pages 147--161. Soc. Math. France, Paris, 2010.

\bibitem{CortezPetite2014}
Mar\'{i}a~Isabel Cortez and Samuel Petite.
\newblock Invariant measures and orbit equivalence for generalized {T}oeplitz
  subshifts.
\newblock {\em Groups Geom. Dyn.}, 8(4):1007--1045, 2014.

\end{thebibliography}
